\documentclass[11pt,twoside,reqno]{amsart}
\usepackage{amsmath}
\usepackage{amssymb}
\usepackage{graphicx}
\usepackage{mathrsfs}
\usepackage[mathcal]{eucal}
\usepackage[alphabetic]{amsrefs}
\usepackage{array}
\BibSpec{article}{%
    +{}  {\PrintAuthors}                {author}
    +{,} { \textit}                     {title}
    +{.} { }                            {part}
    +{:} { \textit}                     {subtitle}
    +{,} { \PrintContributions}         {contribution}
    +{.} { \PrintPartials}              {partial}
    +{,} { }                            {journal}
    +{}  { \textbf}                     {volume}
    +{}  { \PrintDatePV}                {date}
    +{,} { \issuetext}                  {number}
    +{,} { \eprintpages}                {pages}
    +{,} { }                            {status}
    +{}  { \parenthesize}               {language}
    +{}  { \PrintTranslation}           {translation}
    +{;} { \PrintReprint}               {reprint}
    +{.} { }                            {note}
    +{.} {}                             {transition}
}

\BibSpec{book}{%
    +{}  {\PrintPrimary}                {transition}
    +{,} { \textit}                     {title}
    +{.} { }                            {part}
    +{:} { \textit}                     {subtitle}
    +{,} { \PrintEdition}               {edition}
    +{}  { \PrintEditorsB}              {editor}
    +{,} { \PrintTranslatorsC}          {translator}
    +{,} { \PrintContributions}         {contribution}
    +{,} { }                            {series}
    +{,} { \voltext}                    {volume}
    +{,} { }                            {publisher}
    +{,} { }                            {organization}
    +{,} { }                            {address}
    +{,} { \PrintDateB}                 {date}
    +{,} { }                            {status}
    +{}  { \parenthesize}               {language}
    +{}  { \PrintTranslation}           {translation}
    +{;} { \PrintReprint}               {reprint}
    +{.} { }                            {note}
    +{.} {}                             {transition}
}

\BibSpec{collection.article}{%
    +{}  {\PrintAuthors}                {author}
    +{,} { \textit}                     {title}
    +{.} { }                            {part}
    +{:} { \textit}                     {subtitle}
    +{,} { \PrintContributions}         {contribution}
    +{,} { \PrintConference}            {conference}
    +{}  {\PrintBook}                   {book}
    +{,} { }                            {booktitle}
    +{,} { \PrintDateB}                 {date}
    +{,} { pp.~}                        {pages}
    +{,} { }                            {status}
    +{}  { \parenthesize}               {language}
    +{}  { \PrintTranslation}           {translation}
    +{;} { \PrintReprint}               {reprint}
    +{.} { }                            {note}
    +{.} {}                             {transition}
}

\usepackage{hyperref}

\newtheorem{Thm}{Theorem}
\newtheorem{Prop}[Thm]{Proposition}
\newtheorem{Lem}[Thm]{Lemma}
\newtheorem{Cor}[Thm]{Corollary}
\numberwithin{equation}{section}
\numberwithin{Thm}{section}

\theoremstyle{definition}
\newtheorem{Def}[Thm]{Definition}
\newtheorem{Ex}[Thm]{Example}
\newtheorem{Rem}[Thm]{Remark}

\def\le{\leqslant}
\def\ge{\geqslant}
\def\C{{\mathbb C}}
\def\R{{\mathbb R}}

\def\Z{{\mathbb Z}}
\def\N{{\mathbb N}}
\def\Re{{\operatorname{Re}}}
\let\parasymbol=\S
\def\secref#1{\parasymbol\ref{#1}}
\def\S{\varSigma}
\def\ssm{\smallsetminus}
\def\d{\mathrm d}

\def\l{\lambda}

\def\f{\varphi}
\def\^#1{\widehat{#1}}

\def\:{\colon}

\def\<{\left<}
\def\>{\right>}
\def\~#1{\tilde #1}

\def\pd#1#2{\frac{\partial #1}{\partial #2}}

\def\Rbar{\overline{\mathbb R}}
\def\argmax{\operatornamewithlimits{argmax}}
\def\B{\mathscr B}
\newcommand{\oplush}{\mathbin{\mathop{\oplus}\limits_{\scriptscriptstyle h}}}
\newcommand{\otimesh}{\mathbin{\mathop{\otimes}\limits_{\scriptscriptstyle h}}}

\def\Sp{\operatorname{Sp}}
\def\GL{\operatorname{GL}}

\def\Wloc{W^-_{\mathrm{loc}}}
\def\Wglo{W^-_{\mathrm{glob}}}

\makeatletter
\newcommand{\idint}{%
  \mathop{%
    \vcenter{%
      \hbox{%
        \ooalign{%
          \hfil\scalebox{1.90}{$\displaystyle\bigcirc$}\hfil\cr
          \hfil\kern1.0pt
\raisebox{.55ex}{$\textstyle\int$}%
\kern.7pt\hfil\cr
        }%
      }%
    }%
  }\displaylimits
  \@ifnextchar_{\idint@sub}{}%
}

\newcommand{\idints}{%
  \mathop{%
    \hbox{%
      \raise0pt\hbox{%
        \scalebox{.7}{%
          \ooalign{%
            \hfil\scalebox{1.90}{$\displaystyle\bigcirc$}\hfil\cr
            \hfil\kern1.0pt
              \raisebox{.55ex}{$\textstyle\int$}%
            \kern.7pt\hfil\cr
          }%
        }%
      }%
    }%
  }\nolimits
  \@ifnextchar_{\idints@sub}{}%
}

\def\idint@sub_#1{%
  _{\raisebox{-5pt}{$\scriptstyle #1$}}%
}

\def\idints@sub_#1{%
  _{\raisebox{-5pt}{$\scriptstyle #1$}}%
}

\makeatother

\begin{document}

\title{Bellman semigroups. Quarante ans d'apr\`es}

\author{S. Yakovenko}

\address{Weizmann Institute of Science\\
Rehovot, Israel}
\email{\textup{\texttt{sergei.yakovenko@weizmann.ac.il}}
}

\begin{abstract}
We revisit the relation between discrete- and continuous-time stationary dynamic optimization, using both the Bellman and Hamiltonian descriptions. We study the problem of embedding discrete Bellman--Lax--Oleinik operators and the associated Hamiltonian cascades into continuous-time dynamics. Rather subtle obstructions to such embeddings emerge even under the most natural assumptions.
\end{abstract}

\dedicatory{Remembering Zvika}


\date{\today}

\maketitle
\section*{Zvika Artstein. In memoriam}
\begin{Small}
\subsection*{Ad personam}
I started working at the Weizmann Institute in October 1991, when Zvika was Dean of the Faculty of Mathematics. In effect, this was my first glimpse into the inner machinery of Western academia, so strikingly different from Soviet academia. Zvika took it upon himself to be my Virgil, although he had no formal obligation whatsoever to do so.

Not long afterwards, he dropped by my office and, almost in passing, asked what I was planning to work on during the two years I was to spend at Weizmann. At the time I had two research subjects rather far removed from one another. One was in the analytic theory of differential equations; the other concerned stationary optimization problems over long and infinite time horizons. Zvika and his students, notably Arie Leizarowitz, who later became a professor at the Technion, were actively working on infinite horizon optimization, although Zvika certainly knew the former subject very well and was fully aware of my ``starting position'' there. When I told him that I intended to keep working quietly on both, he replied with a degree of bluntness quite unusual for him—Zvika was normally very gentle in such matters: "Sergei, do what you can do better than others. Count limit cycles, and don't get distracted by working in an overcrowded field!"

By that time I had not yet had much opportunity to appreciate just how overcrowded that field was. Most journals had been exceedingly difficult to obtain in Soviet libraries, while I had not yet fully discovered the new wealth of resources suddenly open to me. But I trusted Zvika immediately. I put all my optimization drafts into a drawer and concentrated entirely on the analytic theory of differential equations.

Zvika's prescription worked perfectly. In due course, a two-year contract turned into a tenured professorship and three terms as Head of the Department of Mathematics. Yet throughout all those years I never quite lost the regret that I had abandoned several unsolved problems halfway through. From time to time, meeting colleagues at conferences, I would tell them about these problems. They would politely nod, agree that the questions were very interesting and, to the best of their knowledge, still open—but nothing ever came of these conversations.

Zvika passed away on October 31, 2025. His closest colleagues and friends decided to organize a special journal issue in his memory and invited me to contribute a paper. It was then that I realized that I owed it to my late friend to return, if only briefly, to ``his'' subject, to pay tribute to him, and at the same time to settle some unfinished business of my own.

It was something of a mission impossible. Imagine Rip Van Winkle, having slept through two consecutive twenty-year naps, deciding to survey what had happened in his absence and then, for good measure, to finish what he had left undone before falling asleep forty years earlier. The task would have been utterly beyond me without my faithful assistant Geppy (ChatGPT), with whom, for about half a year, I discussed this paper for several hours almost every day. The result is for the reader to judge: \emph{feci quod potui, faciant meliora potentes}. For me, this paper is, above all, a way of repaying a debt to a colleague, a friend of many years, and a remarkable man.

\subsection*{Ad rem}

Some forty years ago, while studying stationary dynamic optimization problems on an infinite time horizon, I explored several approaches to defining infinite extremals in the simplest, discrete-time setting. All roads led to Rome: the spectral equation for the Bellman operator. The discrete Euler–Lagrange equations were easy to derive, but their qualitative analysis remained elusive. During those same years, before coming to Weizmann, I also turned to the corresponding continuous-time variational problem and found, in the convex-concave setting particularly relevant to mathematical economics, a geometric interpretation of the missing boundary conditions at infinity: an invariant Lagrangian manifold passing through a hyperbolic stationary point of the Hamiltonian flow.

What remained unfinished was the attempt to reconcile these two approaches. One possibility was to develop a Hamiltonian theory for the discrete Euler–Lagrange equations, reproducing as much as possible of the continuous-time symplectic picture. Another was to proceed in the opposite direction, embedding discrete-time dynamics into continuous-time flows or Bellman semigroups, thereby bridging the gap between the discrete and continuous theories. Neither program was carried out at the time. The present paper returns to this unfinished business, forty years later, to explore how far such a reconciliation is possible—and where it breaks down.
\par\end{Small}

\section{Stationary discrete time dynamic optimization problems on infinite intervals}\label{sec:infinite-extremals}

\subsection{What do we optimize?}
Simple optimization problems asking to describe, say, maxima of a function $f\:X\to\R$ on the given domain $X\subseteq\R^n$, appear from the first day of Analysis. In modern terms,  they involve various assumptions on the function $f$ (semicontinuity, differentiability, smoothness, concavity, \dots) and on the domain $X$ (finite-dimensionality, topology, compactness, etc.). We will freely use combinations of these properties, referring to them in informal language. 

A \emph{stationary discrete time dynamic optimization problem} is a shorthand for the following situation. We are given a phase space $X$ (everywhere below---a subset of a finite-dimensional Euclidean space $\R^n$) and extended-real-valued function $b\:X\times X\to \Rbar=\R\cup\{-\infty\}$, called the \emph{discrete Lagrangian}. A \emph{trajectory} is a sequence of points $\{x_0,x_1,\dots,x_t,\dots\}$, $t\in\mathbb Z_+$. For a finite trajectory with the horizon $T\in\N$, $T<+\infty$, the sum 
\begin{equation}\label{finite-sum}
S_T(x_0, x_1,\dots, x_T)=\sum_{t=0}^{T-1} b(x_t,x_{t+1})
\end{equation} 
is well defined in $\Rbar$. Under standard compactness and semicontinuity assumptions, its maximization is well posed; throughout this introductory discussion we will tacitly or explicitly assume such conditions whenever needed: you should not bear in mind anything more involved than the problem of maximization of a well-behaved function over a finite-dimensional compact $\underbrace{X\times\cdots\times X}_{T\text{ times}}$.

\begin{Rem}[important]
Unless explicitly stated otherwise, all discrete-time optimization problems considered below are stationary: the transition Lagrangian \(b(x,y)\) does not depend explicitly on time.
\end{Rem}

The optimality condition with respect to local variations requires that for any three consecutive states $x_{t-1},x_t,x_{t+1}\in\R^n$ the middle point $x_t$ satisfies the local optimization problem
\begin{equation}\label{local-discrete}
x_t \in \argmax_{z\in X} b(x_{t-1}, z)+b(z, x_{t+1}) \qquad \forall t=1,2,\dots,T-1
\end{equation}
(we tacitly assume that the maximum exists and is not equal $\pm\infty$). 
Assuming $b$ differentiable in both arguments and
$x_t\in\operatorname{int}X$, the standard Fermat necessary
condition for the extremum takes the form
\begin{equation}\label{DEL}
\partial_y b(x_{t-1}, x_t)+\partial_x b (x_t, x_{t+1}) = 0\qquad\forall t=1,\dots,T-1.
\end{equation}
This equation is sometimes referred to as the \emph{discrete Euler-Lagrange equation} (see below a brief synopsis of the continuous-time case).

The system \eqref{DEL} is a second-order difference equation involving every three consecutive points of the trajectory. To specify the boundary conditions for these equations, one has to do some extra work. From the point of view of Mathematical Economics, for which such problems are endemic, these boundary conditions are asymmetric. Fixing the initial condition $x_0$ is quite natural. The ``mechanical'' condition of fixing $x_T$ for a finite $T$ is much less natural. Instead, the economists prefer to add a terminal bonus $f(x_T)$ to the criterion of the finite-horizon optimization which takes then the form
\begin{equation}\label{finite-T}
\sum_{t=0}^{T-1} b(x_t, x_{t+1}) + f(x_T)\to\max\quad\text{subject to } x_0=\text{fixed}.
\end{equation}  
This formulation seems impeccable so that the solution exists under ``reasonable technical assumptions'' on the discrete Lagrangian and the terminal bonus. If the terminal state $x_T$ is interior to $X$, the
boundary conditions for this finite-dimensional optimization
problem take the form
\begin{equation}\label{DEL-bv}
x_0 \text{ is prescribed}, \qquad \partial_y b(x_{T-1},x_T)+\partial_x f(x_T)=0.
\end{equation}
These boundary conditions are traditionally called \emph{conditions of transversality}. 

\begin{Rem}
Note two key circumstances. First, the system of equations \eqref{DEL} is only an implicit relation imposed on any three points $(x_{t-1},x_t,x_{t+1})$. It is not explicitly resolved giving $x_{t+1}$ as a function of $(x_{t-1},x_t)$. In other words, despite its independence of $t$ (stationarity) it in general does not define the map $X\times X\to X\times X$, $(x_{t-1},x_t)\mapsto (x_t,x_{t+1})$ which is normally associated with a (nonlinear) second order difference equation.  

Second, the boundary value problem is rather nontrivial even if the \eqref{DEL} can be explicitly resolved with respect to $x_{t+1}$. Unlike the standard initial conditions which for the second order difference equation require knowledge of $(x_0,x_1)\in X^2\subseteq\R^{2n}$, it imposes $n$ conditions on $x_0$ (the initial point) and $n$ other conditions on the terminal point $x_T$. There are no automatic theorems that guarantee existence of solutions for the corresponding boundary value problem \eqref{DEL-bv}.

In the continuous-time case (see below), the first difficulty disappears under the natural strict concavity assumption on the Lagrangian with respect to the velocities (negative definiteness of the Hessian $L_{vv}$, standard in the calculus of variations). The second difficulty persists.
\end{Rem}

\subsection{Infinite horizon: the troubles in the making}
The finite-horizon discrete optimization problems arising in economic
models suffer from a common feature, sometimes described as the
``Apr\`es nous, le d\'eluge'' phenomenon. Very roughly, during most of
the period $[0,T]$ an optimal trajectory maintains a certain balance
between investment and consumption, investment being necessary to
ensure a higher level of future consumption. Near the terminal time,
however, this balance is radically shifted in favor of consumption:
the finite-horizon optimization criterion ``sees no future'' beyond
$T$, and therefore assigns no value to investment whose returns would
occur after $T$. This artificial terminal effect suggests moving the
horizon to $T=+\infty$.

However, such a formal extension destroys the very formulation of the
problem. The finite sum becomes an infinite series, which may diverge
on some or even all admissible trajectories. At the same time, we lose
the possibility of introducing a terminal bonus as in
\eqref{finite-T}. 

There are several natural ways out of this conundrum.
Remarkably, they all lead to the same functional equation, which will
eventually provide us with a definition of infinite-horizon extremals.  The exposition below follows \cite{dan1990} and \cite{bellman}.

\subsubsection{Calibration transformation of the discrete Lagrangian}\label{sec:nil-calibration}
Consider an example where the discrete Lagrangian has a very special form (we call it \emph{nil-Lagrangian}),
\begin{equation}\label{b-le-0}
\forall x,y\in X \quad b(x,y)\le 0,\quad\text{and}\quad \forall x\in X\ \exists y\in X:\ b(x,y)=0.
\end{equation} 
Then all problems related to the possible divergence of the series $S=\sum_{t=0}^{+\infty}b(x_t,x_{t+1})$ disappear completely. For any initial condition $x_0$ and any trajectory the series either converges monotonically to $S\le 0$, or diverges to $-\infty$ by the first assumption. The value $S=0$ is always feasible by the second assumption, and construction of the optimal trajectory is trivial: 
$$
x_{t+1}\in\argmax_{y\in X} b(x_t, y),\qquad\forall t=0,1,2,3,\dots.
$$ 
Of course, not all Lagrangians are nil-Lagrangians. Yet there is an equivalence relation between Lagrangians which does not change the discrete Euler-Lagrange equations. 

\begin{Def}\label{def:cohomol}
Two Lagrangians $b'(x,y)$ and $b''(x,y)$ defined on the same space $X\times X$ are called \emph{cohomologous}, if there exists a function $s\:X\to\R$ and a number $\l\in\R$ such that
\begin{equation}\label{cohomol}
\forall x,y\in X\quad b''(x,y)=b'(x,y)+s(y)-s(x)-\lambda
\end{equation}
(the reason why $\l$ occurs with the sign minus, will be clear later, in \secref{sec:reappear}).
\end{Def}
This is indeed an equivalence relation. Its role is fully revealed by the following elementary property.

\begin{Prop}
If $b'$ and $b''$ are two cohomologous Lagrangians, then for any $T<+\infty$ and for any $p,q\in X$ the two optimization problems
$$
\sum_{t=0}^{T-1} b'(x_t,x_{t+1})\to\max,\qquad\sum_{t=0}^{T-1} b''(x_t,x_{t+1})\to\max
$$
with the same boundary conditions $x(0)=p,x(T)=q$, have the same extremals. 
\end{Prop}

\begin{proof}
The difference between the two sums is equal to the telescopic sum $\sum_{t=0}^{T-1} \bigl( s(x_{t+1})-s(x_t)\bigr)-\lambda T=s(q)-s(p)-\lambda T$, that is, a constant. Two functions which differ by a constant have the same supply of extremal trajectories.
\end{proof}

This raises the obvious question: \emph{under what conditions a given discrete Lagrangian is cohomologous to a nil-Lagrangian}? 

\subsubsection{Dynamic programming}
The above construction, if properly implemented, would give a good definition of an actually infinite trajectory. One can instead consider a family of finite-horizon optimization problems that would, in a sense, be independent of the planning horizon $T$. The corresponding construction is known as the dynamic programming method.

Consider a finite horizon problem \eqref{finite-T} with a terminal bonus $f\:X\to\Rbar$. The dynamic programming method consists in the following ``reduction backwards''. Solution of the problem \eqref{finite-T} on the interval $[0,T]$ is fully equivalent to solution of the optimization problem 
\begin{equation}\label{T-1}
\sum_{t=0}^{T-2} b(x_t,x_{t+1}) + g(x_{T-1}),
\end{equation}
where $g(\cdot)$ is a new terminal bonus corresponding to the last step of the trajectory, cf.~with \eqref{DEL-bv}. 
$$
g(x_{T-1})=\max_{y\in X} \bigl(b(x_{T-1}, y) + f(y)\bigr).
$$
This construction can be obviously iterated as follows.

\begin{Def}[Bellman operator]\label{def:bellman}
Let $b\:X^2\to\Rbar$ be a discrete Lagrangian. Let $\Rbar(X)$ denote the set of all functions $X\to\Rbar$, without a priori regularity or boundedness requirements. Whenever the supremum below belongs to $\Rbar$, it defines the Bellman operator by 
\begin{equation}\label{bellman}
\B\:\Rbar(X)\to\Rbar(X), \qquad g=\B f \iff g(x)=\sup_{y\in X}\bigl(b(x,y)+ f(y)\bigr). 
\end{equation}
We will sometimes call the function $b(x,y)$ the \emph{kernel} of the operator $\B$. The operator $\B$ itself will be from now on called the \emph{Bellman operator} associated with the kernel (discrete Lagrangian) $b$. 
\end{Def}

\begin{Rem}
In general, the supremum in \eqref{bellman} need not be attained or finite. In the optimization problems considered below we assume, whenever the formula is used, that it is attained and belongs to $\Rbar$. Accordingly, we shall write $\max$ in place of $\sup$ as a reminder of this standing assumption.
\end{Rem}

In these terms the two finite-horizon problems, \eqref{finite-T} and another one, 
\begin{equation}\label{bellman-one-down}
\sum_{t=0}^{T-2} b(x_t, x_{t+1}) + g(x_{T-1}), \qquad g=\B f,
\end{equation}
will be fully equivalent in the most natural sense: any solution of the $T$-horizon problem restricted on the shorter interval will yield a solution of the ($T-1$)-horizon problem and conversely, any solution of the $(T-1)$-horizon problem can be automatically extended to yield solutions of the $T$-horizon problem. 

When iterated \emph{ab ultimo}, this scheme allows constructing a sequence of functions $f,\B f,\B^2 f, \dots \B^{T-1}f$ (backwards iterations) that reduce the initial $T$-step optimization problem to the one-step problem 
$$
x_1 \in\argmax_{y\in X} \bigl(b(x_0, y)+ \B^{T-1}f(y)\bigr)
$$ 
that can be continued by induction. 

\subsubsection{How this could be put to work?}
Assume that we have found a function $f^*\in\Rbar(X)$ which is fixed by the Bellman operation with a given kernel $b$, $\B f^*=f^*$. Then all finite-horizon problems will become self-consistent: the inductive rule 
$$
x_{t+1}\in\argmax_{y\in X} \bigl(b(x_t,y)+ f^*(y)\bigr), \qquad t=0,1,2,3,\dots
$$
will provide solutions to the respective finite-horizon problems for an arbitrary $T$ with the same terminal bonus $f^*$. In plain words, the terminal bonus $f^*$ as above, can automatically recede to the future without affecting all preceding finite horizon problems. The finite horizon problem with such bonus becomes insensitive to the parameter $T$ as soon as it takes any finite value. 

What remains to notice is that adding a constant term changes nothing in the behavior of the respective extremals (see above). This allows us to replace the assumption on the function $f^*\in\Rbar(X)$ by a weaker assumption involving an arbitrary additional parameter $\l$,
\begin{equation}\label{eigenB}
\B f^*= f^*+\lambda, \qquad f^*\in\Rbar(X),\ \lambda\in\R.
\end{equation}  

\begin{Def}
The functional equation \eqref{eigenB} with respect to the unknown function $f^*\in\Rbar(X)$ and the number $\l\in\R$ will be called the \emph{spectral equation} for the operator $\B$ for the reasons to be seen very shortly. An eigenfunction is required to be nontrivial, i.e., $f^*\not\equiv-\infty$; an eigenvalue is a number $\lambda$ admitting such an eigenfunction.
\end{Def}

\subsubsection{How to bridge the two approaches}
The approach based on finding a nil-Lagrangian $b''$ cohomologous to the given kernel $b(x,y)=b'(x,y)$ would allow us to change the Lagrangian (without changing its extremals) so that the infinite sum $\sum_{t=0}^{\infty} b''(x_t,x_{t+1})$ makes sense in $\Rbar$ and the extremal value of this sum is equal to zero. Speaking in metamathematical terms, such trajectory is constructed in the class of \emph{actually infinite paths}. 

The second approach rather focuses on constructing potentially infinite extremals. Fixing any solution of the spectral equation $f^*$, we can for any initial point $x_0$ and any $T$ construct (possibly, multiple) $f^*$-optimal $T$-step extremals. Finite pieces of $f^*$-extremals with the same $f^*$ can be arbitrarily stitched together: if $\{x_t\}_{t=0}^T$ is one such extremal with $x_0=p\in X$, and $\{x'_t\}_{t=0}^{T'}$ is another extremal with the initial condition $x_0'=x_T$, then the stitched $T+T'$-step trajectory will be $f^*$-optimal with the initial condition $p$.

\subsubsection{Closing the loop: the spectral equation reappearing}\label{sec:reappear}
To close the loop, there remains to make one simple step in \secref{sec:nil-calibration}. If $b''$ in \eqref{cohomol} is a nil-Lagrangian, then the condition \eqref{b-le-0} can be put into one simple form, $\B''0=0$, where $\B''$ is the associated Bellman operator. Then \eqref{cohomol} yields the tautological identity 
\begin{align*}
\forall x\in X\quad 0&=\max_y b''(x,y) 
\\ &= \max_{y}\bigl(b'(x,y)+s(y)-s(x)-\l\bigr)
\\
&= \max_y\bigl(b'(x,y)+s(y)\bigr)-s(x)-\l. 
\end{align*}
This means that if the Lagrangian $b'$ is cohomologous to the nil-Lagrangian $b''$ as in \eqref{cohomol} and $\B'$ is the Bellman operator associated with $b'$, then we have the identity
\begin{equation}\label{bellman-coh}
\B' s= s+\lambda, \qquad s\:X\to\R, \ \l\in \R.
\end{equation}

\subsection{The spectral equation for the Bellman operators}\label{sec:spectral}
Given a discrete Lagrangian $b\:X\times X\to\Rbar$, there is no intuitive way to define infinite extremals $\{x_t\in X:t\in\Z_+\}$ via simple renormalization of the functional. The classical optimal-growth economics works with different alternatives using more involved notions of \emph{overtaking} and \emph{catching up}, binary relationships on infinite admissible trajectories with the same initial point, which not always exist. Among the contributions to this approach were the works of Arie Leizarowitz, a student and collaborator of Zvika Artstein, on the existence of overtaking-optimal trajectories in infinite-horizon control problems, see \cite{Leizarowitz1985}.

Instead, two rather different approaches motivate the same definition. Such infinite extremals can be associated with \emph{any} solution $f^*$ of the spectral equation \eqref{eigenB} for the nonlinear Bellman operator \eqref{bellman}. Among other features, at interior states where $b$ is differentiable, these infinite trajectories satisfy the discrete Euler--Lagrange equations \eqref{DEL}. Indeed, the usual Fermat condition applies to every interior triple point $(x_{t-1},x_t,x_{t+1})$ of an optimal finite segment. At boundary states, only the constrained first-order optimality condition is available. 

\begin{Def}[main]\label{def-f-extremal}
Let $\B\:\Rbar(X)\to\Rbar(X)$ be a Bellman operator and $f^*\in\Rbar(X)$ a solution of the spectral equation $\B f^*=f^*+\l$. For an initial point $x_0$ satisfying $f^*(x_0)>-\infty$, an \emph{infinite $f^*$-extremal} is any sequence of points $x_t\in X$ built inductively using the rule
$$
x_{t+1}\in\argmax_{y\in X} b(x_t,y) + f^*(y), \qquad t=0,1,2,\dots.
$$
\end{Def} 

\begin{Rem}
Definition~\ref{def-f-extremal} is highly conditional, first and foremost on the existence of solutions of the spectral equation associated with the Lagrangian $b(x,y)$. To guarantee the existence, one must impose certain conditions (compactness of $X$, properness and upper semicontinuity of $b$ as an extended-real-valued function on $X\times X$). The inductive rule also presumes certain regularity of the solution $f^*$ provided that it exists. As this semi-popular treatise is not meant to be a comprehensive monograph exploring all meaningful combinations of assumptions on $X$ and $b$, I will restrict myself only to the examples most simple technically and best motivated economically. For instance, the set-valued operation $\argmax$ might well be ill defined if the solution $f^*$ lacks some properties like boundedness or upper semicontinuity.

A completely different intrigue unfolds in the case where the spectral equation has multiple solutions. Then anyone interested in applications of this theory would have to decide, which solutions $f^*$ are better suited for the applications. 

In particular, consider the finite horizon problem \eqref{finite-T} in the case where the terminal bonus $f$ is \emph{not a solution} to the spectral equation (if it were, the optimal $T$-step solutions simply stabilize to the infinite $f^*$-extremal in the sense that any such $T$-step path is the initial $T$-segment of an infinite $f^*$-extremal and vice versa). Do the finite horizon $T$-step extremals $\{x_{t,T}:t=0,1,\dots,T\}$ for the problem \eqref{finite-T} converge in any sense to an infinite $f^*$-extremal? If yes, then in what sense and how $f^*$ should be chosen if the spectral equation has multiple solutions? 

In classical optimal-growth models of economics, this circle of questions is closely related to the so-called \emph{turnpike theory}. In its simplest stationary form, the typical picture is the following. There is a unique universal (usually, stationary or, in the case where $b$ is first order homogeneous in $(x,y)$, a geometric progression $x_t^*=\mu^t x_*$), called \emph{the golden rule path}.  Then almost all $T$-step trajectories first reach the golden rule path from the prescribed initial state $x_0$, follow this path until almost the moment $T$ but at the very end of the allotted time switch to the ``apr\`es nous, le d\'eluge'' tactics, depending on the initial terminal bonus. 

Part of these questions was answered in the recent decades. \emph{Hanc marginis exiguitas non caperet}.
\end{Rem}

\subsection{Idempotent algebra and analysis}
The specific form of the spectral equation (and even the choice of the name for it) is not accidental. In the mid-1980s several people independently made the following observation: the extended real line $\Rbar$ becomes a semiring with two algebraic operations, $\max\{a,b\}$ and $a+b$. These two operations satisfy axioms of a semiring being commutative and distributive, with $-\infty$ being the $\max$-neutral element and $0$ being the neutral element by the addition. It is not a ring for one simple reason: the operation $(a,b)\to\max\{a,b\}$ does not always admit the inverse: the equation $\max\{a,x\}=b$ has no solution for $b<a$ and has infinitely many solutions for $b=a$. It is called an \emph{idempotent} semiring because $\max{a,a}=a$.

\begin{Def}[max-plus semiring]\label{def:semiring}
The set $\Rbar=\R\cup\{-\infty\}$ with the operations 
\begin{equation}\label{max-plus}
a\oplus b=\max\{a,b\},\ a\otimes b= a+b,\quad a\oplus{(-\infty)}=a,\ a\otimes{(-\infty)}=-\infty
\end{equation} 
is called the \emph{max-plus idempotent semiring}.
\end{Def}

\begin{Small}
\begin{Rem}
The max-plus semiring may replace the genuine ring (or the field $\R$) for many algebraic constructions. For instance, the space of $\R$-valued functions on an abstract set $X$, which carries the structure of the  $\R$-linear space, becomes a semimodule $\Rbar(X)$ over $\Rbar$. To stress that such objects differ from their $\R$-linear analogs, they are sometimes used as tropical objects (e.g., tropical matrices, to distinguish them from the real matrices) etc. 

The tropical max-plus semiring introduced in Definition~\ref{def:semiring} is sometimes called itself the tropical semiring, although this term is more often applied to the isomorphic min-plus semiring on $\underline\R=\R\cup\{+\infty\}$.  The pair of operations max-times also yields an isomorphic semiring. Yet there are non-isomorphic semirings with max-min operations (both idempotent) and the Boolean semiring (on two elements $\{0,1\}$ with $\oplus=\lor$, $\otimes=\land$, or with set theoretic operations $\cup$ and $\cap$ satisfying de Morgan laws).  Throughout this paper we work exclusively with the max-plus idempotent structure.
\end{Rem}
\end{Small}

\subsubsection{The Maslov dequantization}\label{sec:dequantization}
The idempotent structure is in fact a singular limit of the usual commutative semiring structure on $\R_+$. Consider the logarithmic chart $S\:u\mapsto h\ln u$ in which $h>0$ is a small parameter (``the Planck constant''). It is bijective between $\Rbar$ and $\R_+$ and so various structures can be carried both ways (the inverse map takes $x\mapsto \exp \frac xh$). Define two operations,   
\begin{equation}\label{semiring}  
a \oplush b= h\ln \bigl(\exp \tfrac ah +\exp \tfrac bh\bigr), \qquad a\otimesh b = h\ln\bigl(\exp \tfrac ah \cdot \exp \tfrac bh\bigr)=a+b.
\end{equation}
These operations transport the usual ``linear'' structure from $\R_+$ onto $\Rbar$, with $h\ln 0=-\infty$. Note that there is no inverse operation for the usual addition $+$ on $\R_+$, so the word ``linear'' should be enclosed in quotes: although $\R$ is the ring, $\R_+$ is only a semiring with respect to $+,\times$. 

The first operation easily admits passage to the limit as $h\to0+$: $$\lim_{h\to0_+} a\oplush b=\max\{a,b\}= a\oplus b,$$ the second does not depend on $h$ at all, being the usual sum. 

\begin{Rem}
The above singular limit was known in mathematical physics long before the formal discovery of the idempotent structure on $\Rbar$. Consider the \emph{linear} parabolic equation $\partial_t u = \frac h2 \partial^2_{zz}u$ with a small parameter $h >0$. Then the Cole--Hopf change of variables $S=h\ln u$ transforms it to the equation $$\partial_t S-\tfrac12 (\partial_z S)^2 = \tfrac h2 \partial^2_{zz} S$$ which is (up to the sign) the viscous Hamilton--Jacobi equation with the viscosity parameter $h>0$. As $h\to0+$, the right hand side disappears and we arrive at the standard Hamilton--Jacobi equation. 
\end{Rem}

In the opposite direction, the viscosity term of the second order provides a regularization of the Hamilton--Jacobi equation. The limit $h\to0+$ is the classical vanishing-viscosity procedure; its algebraic counterpart, in which the ordinary linear structure degenerates into the max-plus one, is usually referred to as the \emph{Maslov dequantization}, see \cite{vasya}, \cite{litvinov}.

\begin{Def}[Maslov dequantization of the max-plus semiring]\label{def:dequantization}
The semiring $\Rbar$ with the operations $\oplush$ and $\otimesh$ is called the Maslov dequantization of the max-plus semiring $\Rbar$.
\end{Def}
The term ``Maslov dequantization'' will be used by extension to various constructions over the max-plus semiring $\Rbar$. 

\subsection{Characteristic property of the Bellman operators}

The idempotent operations on $\Rbar$ induce pointwise operations
$f\oplus g$ and $\mu\otimes f$ on $\Rbar(X)$. Every Bellman
operator satisfies
\begin{equation}\label{id-linearity}
 \B(f\oplus g)=\B f\oplus\B g,\qquad
 \B(\mu\otimes f)=\mu\otimes\B f.
\end{equation}
The same identity holds for arbitrary pointwise suprema whenever
these suprema belong to the chosen function space. Conversely,
on the full space of $\Rbar$-valued functions, any operator preserving
arbitrary $\oplus$-sums and scalar multiplication is a Bellman operator:
indeed, writing
\begin{equation*}
 \delta_y(x)=\begin{cases}0,&x=y,\\-\infty,&x\ne y,\end{cases}
 \qquad f=\bigoplus_{y\in X}f(y)\otimes\delta_y,
\end{equation*}
we obtain
\begin{equation*}
 (\B f)(x)=\bigoplus_{y\in X}
 \bigl(b(x,y)\otimes f(y)\bigr),\qquad
 b(x,y)=(\B\delta_y)(x).
\end{equation*}
This converse should be applied with some qualification to spaces
of upper semicontinuous functions bounded from above, which need not be
closed under arbitrary pointwise suprema.

The spectral equation takes the form
\begin{equation}\label{spectral}
 \B f^*=\l^*\otimes f^*,\qquad
 f^*\in\Rbar(X),\quad\l^*\in\R.
\end{equation}
For a finite set $X=\{p_1,\dots,p_n\}$, the Lagrangian becomes
a max-plus matrix $(b_{ij})$, and the equation reads
\begin{equation*}
 (\B f)_i=\bigoplus_{j=1}^n(b_{ij}\otimes f_j)
 =\l\otimes f_i.
\end{equation*}
This explains the terminology \emph{eigenfunctions} and
\emph{eigenvalues} for solutions of the spectral equation.

\begin{Small}
\begin{Rem}
For a finite-valued matrix $b=(b_{ij})$, Maslov dequantization
provides another route to existence. Put
$A_h=(\exp(b_{ij}/h))$, $h>0$. This is a strictly positive
matrix, so the Perron--Frobenius theorem supplies a positive
eigenvector $u_h$ and an eigenvalue $\rho_h>0$. Normalize
$\max_i h\log(u_h)_i=0$ and set
$f_i^{(h)}=h\log(u_h)_i$, $\l_h=h\log\rho_h$. Then
\begin{equation*}
 \l_h+f_i^{(h)}
 =h\log\sum_{j=1}^n
 \exp\bigl((b_{ij}+f_j^{(h)})/h\bigr).
\end{equation*}
The inequalities
$\min_{i,j}b_{ij}\le\l_h\le\max_{i,j}b_{ij}+h\log n$
and the finiteness of all entries of $b$ give uniform bounds
on the normalized vectors $f^{(h)}$. Passing to a subsequence
and using the log-sum-exp limit yields a finite-valued solution
of \eqref{spectral}. For matrices with entries $-\infty$,
strict positivity is lost and this argument requires additional
care; the graph-theoretic result below handles that case directly.
\end{Rem}
\par\end{Small}

For an arbitrary phase space $X$, the Bellman operator can be
written symbolically as
\begin{equation*}
 (\B f)(x)=\idint_{y\in X}
 \bigl(b(x,y)\otimes f(y)\bigr)
 =\sup_{y\in X}\bigl(b(x,y)+f(y)\bigr).
\end{equation*}
The glyph $\idints$ recalls integration, but idempotent integration
requires no measure: it is simply a supremum. We use $\sup$
here because a maximum need not be attained on a general $X$.

\subsection{Solvability of the spectral equation}

The finite-dimensional analogy suggests that the spectral equation
\eqref{spectral} should be solvable in many natural situations.
We ask when nontrivial eigenfunctions exist, when the eigenvalue is
unique, when the eigenfunction is unique up to an additive constant,
and when eigenfunctions possess additional regularity, such as
finiteness or continuity. These questions were studied extensively;
see \cites{dan1990,bellman,vasya,litvinov}.

\begin{Small}
The identically $-\infty$ operator has no nontrivial eigenfunction.
More generally, if $\B$ is \emph{nilpotent}, i.e.,
$\B^N=\boldsymbol{-\infty}$ for some finite $N$, it has no
nontrivial eigenfunction with a finite eigenvalue. For finite $X$,
this is the only obstruction to existence.

\begin{Prop}\label{prop:romanovskii}
If $\B$ is a non-nilpotent Bellman operator on $\Rbar(X)$
with finite $X$, then it admits at least one eigenvalue.
\end{Prop}

\begin{proof}
The directed graph of finite-valued transitions has paths of
arbitrarily large finite length, since $\B$ is not nilpotent.
As $X$ is finite, it therefore contains a directed cycle.
Among all simple directed cycles choose one, $\gamma$, with
maximal average gain
\begin{equation}\label{average}
 \l_* =\max_\gamma\l(\gamma),\qquad
 \l(\gamma)=\frac1N\sum_{t=0}^{N-1}
 b(x_t,x_{t+1}),\quad x_N=x_0.
\end{equation}
Put $b^*(x,y)=b(x,y)-\l_*$. Every directed cycle then
has nonpositive total $b^*$-weight, and $\gamma$ has weight zero.
Fix $x_*\in\gamma$ and define $g(x)$ as the supremum of the
$b^*$-weights of all finite admissible paths from $x$ to $x_*$,
including the path of length zero when $x=x_*$. If there is no
such path, put $g(x)=-\infty$.

Deleting cycles of nonpositive weight cannot decrease a path's
weight. Thus the supremum is attained among paths without repeated
vertices (apart from the optional zero-weight return to $x_*$);
there are only finitely many such paths. Hence $g$ is bounded
above and is finite on $\gamma$, so it is not identically
$-\infty$.

For every admissible edge $x\to y$, appending a path from $y$
to $x_*$ shows that $b^*(x,y)+g(y)\le g(x)$; hence
$\B^*g\le g$. If $x\ne x_*$ and $g(x)>-\infty$, the first
edge of a maximizing path gives equality. At $x=x_*$, the
zero-weight cycle $\gamma$ gives equality as well. When
$g(x)=-\infty$, the inequality already forces equality.
Therefore $\B^*g=g$, or equivalently
$\B g=g+\l_*$.
\end{proof}

For infinite $X$, a Bellman operator can be non-nilpotent and
still have no eigenfunction.

\begin{Ex}
Let $X=\Z_+$ and define
\begin{equation*}
 b^-(i,j)=\begin{cases}
 0,&j=i-1,\\
 -\infty,&\text{otherwise}.
 \end{cases}
\end{equation*}
The corresponding Bellman operator is non-nilpotent, but it has
no nontrivial eigenfunction. Indeed, its spectral equation at
$i=0$ forces $f(0)=-\infty$, and the recurrence
$f(i-1)=f(i)+\l$ for $i\ge1$ forces $f\equiv-\infty$.

For contrast, let $b^+(i,j)=0$ when $j=i+1$, and $-\infty$
otherwise. Then the spectral equation is
$f(i+1)=f(i)+\l$, and for every $\l\in\R$ its nontrivial eigenfunctions are
$f(i)=c+\l i$. 
Thus the right-shift operator has infinitely many eigenvalues.
\end{Ex}

\subsubsection{Nontrivial results on solvability of the spectral equation}

These two observations show that the existence of eigenvalues is
a nontrivial issue in the infinite-dimensional case. We list
several positive results in this direction, all known for a long time.

Proposition~\ref{prop:romanovskii} should be compared with the
following much earlier theorem by I. V. Romanovskii.

\begin{Thm}[graph case, \cite{Rom1967}]
Let $\B$ be a Bellman operator on a \emph{finite} state space $X$,
with a \emph{finite-valued} Lagrangian $b\:X\times X\to\R$.
Then the spectral equation \eqref{spectral} admits a finite-valued
solution. The corresponding eigenvalue $\l$ is unique and equal
to the maximal average gain \eqref{average} over all self-avoiding
cycles $\gamma$.

Moreover, the eigenfunctions admit a complete description in
terms of the maximal (critical) cycles of the complete oriented
weighted graph associated with $b$.
\end{Thm}
\par\end{Small}

\begin{Thm}[the compact case, \cite{dan1990}]\label{thm:compact}
If $X$ is a compact topological space and the Lagrangian
$b\in C(X\times X)$ is everywhere finite and jointly continuous
in $x,y$, then there exists a unique eigenvalue $\l$.
Continuous eigenfunctions $f^*\in C(X)$ exist, but need not
be unique. \qed
\end{Thm}

The result extends to the quasi-compact setting, where $\B^N$
is compact for some finite $N$, see \cite{bellman}.
Indeed, if $g\in C(X)$ is a continuous eigenfunction of $\B^N$
with eigenvalue $N\l$, then the ``$\oplus$-average''
\begin{equation}\label{averageing}
f=g\oplus(\B g-\l)\oplus(\B^2g-2\l)
\oplus\cdots\oplus(\B^{N-1}g-(N-1)\l)
\end{equation}
satisfies $\B f=f+\l$.

\begin{Thm}[the strictly concave case, \cite{bellman}]
\label{thm:strictly-concave}
If $X$ is a finite dimensional compact convex subset,
the Lagrangian $b:X\times X\to\R$ is continuous and strictly jointly concave in $x,y$,
and $b(x,x)$ reaches its maximum at an interior point
$x_*\in X$, then $\B$ has a unique eigenvalue
$\l^*=b(x_*,x_*)$, associated with a unique
(modulo an additive constant) finite-valued continuous concave eigenfunction on $X$.
\qed
\end{Thm}

\begin{Small}
Related results had been obtained much earlier by Romanovskii
\cite{Rom1964}. For a continuous transition function $b(x,y)$
on a compact convex state space $X$, he characterized the
asymptotic average gain in terms of closed trajectories and
a dual Bellman-type inequality; under concavity assumptions
he also obtained turnpike-type estimates for finite-horizon
optimal trajectories.

In a different, nonconcave setting, Sobolevskii~\cite{Sobolevskii1999} established uniqueness up to an additive constant of a continuous periodic idempotent Bellman eigenfunction when the associated Aubry--Mather rotation number is irrational.

Investigation of the spectral equation did not stop with
these vintage results, but continued in a different attire.
After the standard exponentiation, passing from
$f\in\Rbar(X)$ to $F=\exp f$, the additive homogeneity $\B(f+c)=\B f+c$
of Bellman operators becomes positive first-order homogeneity, while $\oplus$-additivity remains preservation of maxima.
This brings us, via inverse Maslov dequantization, to the
``nonlinear Frobenius--Perron problem'' \cite{GG2004},
which subsequently developed into a substantial theory.
Further details can be found in the proceedings
\cite{LM2005}, especially \cite{AGW2005}; many of these
results can also be expressed in the language of max-plus
semialgebra, $(\oplus,\otimes)=(\max,+)$.
I do not intend to \emph{catch up} with, let alone
\emph{overtake}\footnote{See \secref{sec:spectral},
p.~\pageref{sec:spectral}.} these developments:
the focus of the present paper lies elsewhere.

\par\end{Small}

\section{Calculus of variations on the infinite interval and the Hamiltonian flow}\label{sec:variations}

We briefly recall the classical calculus of variations in its modern geometric formulation, following Arnold \cite{mmkm}, with particular attention to the Hamiltonian description of extremals and the geometry of invariant Lagrangian manifolds.

\subsection{Continuous time optimization problems: overview}\label{sec:technical}
Let $[0,T]\subseteq\R$ be a finite interval on the real line, and $L\:X\times\R^n\to\R$, $X\subseteq\R^n$, a function called the (genuine) Lagrangian. In parallel with the formulation \eqref{finite-sum} for any sufficiently regular function $x\:[0,T]\to X$ we can form the functional  
\begin{equation}\label{T-integral}
\int_0^T L\bigl(x(t),\dot x(t)\bigr)\,\mathrm dt + f\bigl(x(T)\bigr)
\end{equation}
and look for a trajectory $x(t)$ which \emph{maximizes}\footnote{Usually in the Calculus of Variations (and Mechanics) the problem is to minimize the integral functional as above. Our approach was motivated by problems arising from Mathematical Economics.} the above functional over all admissible trajectories constrained by the initial condition $x(0)=x_0\in X$. 

Contrary to the discrete problem, the variational problem \eqref{T-integral} is \emph{infinite-dimensional} even for finite horizon $T$ (recall that a finite discrete time admissible trajectory is a point in a suitable finite-dimensional set). This requires a number of assumptions on $L$ and $X$ merely to guarantee the existence of solutions. For simplicity, we impose rather restrictive assumptions in order not to be distracted by technicalities.
\begin{enumerate}
 \item $L$ is a smooth function jointly in the variables $x,v$;
 \item In the unconstrained discussion one may take $X=\R^n$; however, it should be stressed that the variables $x$ and $v$ play different roles (state and velocity respectively), so we will retain  the notation $X$ to stress the geometric meaning of some formulas (most notably the notation $T^*X$ for the symplectic manifold which will appear below);
 \item $L(x,\cdot)$ is strictly concave function of $v\in\R^n$;
 \item As $|v|\to+\infty$, the Lagrangian tends to $-\infty$ sufficiently fast, say, $L(x,v)\le C-c|v|^p$, $C,c>0$, $p>1$ uniformly in $x\in X$ (the \emph{coercivity}). 
\end{enumerate} 

\begin{Small}

\subsection{Euler--Lagrange differential equations}
In a way completely similar to that of the discrete case, one can rather easily write down necessary conditions for optimality. They take the form of a second order ordinary differential equation, called the Euler-Lagrange equation: 
\begin{equation}\label{CEL}
\frac{{\mathrm d}}{{\mathrm dt}}\,\partial_v L\bigl(x(t),\dot x(t)\bigr)=\partial_x L\bigl(x(t),\dot x(t)\bigr),
\end{equation}
coupled with the boundary conditions (the second condition is called the condition of transversality),  
\begin{equation}\label{EL-transv}
x(0)=x_0\in X,\qquad \partial_v L(x(t),\dot x(t))+\mathrm df(x(t))\big|_{t=T}=0.
\end{equation}
Equation \eqref{CEL} is a second-order equation not yet resolved with respect to the second derivative $\ddot x(t)$: the left-hand side requires substituting the current values $(x(t),\dot x(t))$ into the partial derivative $\partial_v L(x,v)$ and then taking the total derivative of this composition with respect to $t$: this transforms \eqref{CEL} into the equation
\begin{equation}  
L_{vv}(x,\dot x)\, \ddot x+L_{vx}(x,\dot x)\,\dot x-L_x(x,\dot x)=0. 
\end{equation}
To resolve this equation with respect to $\ddot x$, one has to invert the \emph{symmetric matrix} $L_{vv}$. If it is non-degenerate at all points $(x,v)$ (for instance, negatively definite), then one can resolve \eqref{CEL} with respect to $\ddot x$ and define a continuous time dynamical system on $X\times\R^n$.

\begin{Rem}
The Euler--Lagrange equations \eqref{CEL} is a perfect tool to study finite-horizon optimization problems: one has to find solutions of this equation satisfying two boundary value conditions 
\eqref{EL-transv}. Yet unlike the standard Cauchy problem which asks to determine solution by the initial data $x(0)=x_0$, $\dot x(0)=v_0$, the two-point boundary value problem may well have no solutions or multiple solutions. This question is by no means technical. 
\end{Rem}   

\subsection{Legendre transform and the Hamiltonian systems}\label{sec:Legendre-transform}
The Euler--Lagrange equations \eqref{CEL} are fine, but in 1834-1835 it turned out that they can be rewritten in a different form after using a new variable. This new form, discovered by William Rowan Hamilton, was truly remarkable for several reasons, one of them the rich geometric structure which is behind these equations. 

\begin{Def}
Let $L\:X\times\R^n$, $L=L(x,v)$ be a smooth Lagrangian such that the Hessian matrix $L_{vv}$ is nondegenerate (say, strictly negative definite), and $p\in\R^n$ a new variable (called, depending on the area and up to sign, momentum in Mechanics, prices in Economics, and the linear dual to the variable $v\in\R^n$ in Linear algebra, which justifies a better notation $p\in \R^{n*}$).  

The Legendre transform of the Lagrangian $L$ is the function $H\:X\times\R^{n*}$, $(x,p)\mapsto H(x,p)$, called the \emph{Hamiltonian}, which (under our convention on signs) is defined as
\begin{equation}\label{Ham}
H(x,p)=\<p,v\>+L(x,v),
\end{equation} 
where $v = v(x,p)$ is solution to the algebraic equation 
\begin{equation}\label{implicit-v}
p=-L_v(x,v)=-\pd{L}{v}(x,v).
\end{equation}
The concavity and coercivity assumptions made above ensure global solvability; nondegeneracy of $L_{vv}$ ensures smooth local inversion.
\end{Def}

\begin{Rem}
Assuming that the Lagrangian $L(x,v)$ is concave in $v\in\R^n$ and satisfies the coercivity growth assumptions listed in \secref{sec:technical}, the Legendre transform can be defined without the smoothness/invertibility condition by the Fenchel formula
\begin{equation}\label{Fenchel}
H(x,p)=\max_{v\in\R^n}\bigl(\<p,v\>+L(x,v)\bigr),
\end{equation}
eventually taking extended-real-values from $\R\cup\{+\infty\}$.
\end{Rem}

\begin{Thm}[\cite{mmkm}]
If the trajectory $(x(t),\dot x(t))$ is a solution to the
Euler--Lagrange equation \eqref{CEL}, with
$x(t)\in\operatorname{int}X$, and the corresponding Hamiltonian
is differentiable in both variables, then the trajectory
$(x(t),p(t))$, where
$p(t)=-L_v\bigl(x(t),\dot x(t)\bigr)$,
satisfies the Hamiltonian equations
\begin{equation}\label{HDE}
\left\{
\begin{aligned}
\dot x=\phantom{-}&\pd{H(x,p)}p,\\
\dot p=-&\pd{H(x,p)}x.
\end{aligned}
\right.
\end{equation}
If the trajectory maximizes the action on $[0,T]$
with fixed initial state $x(0)=x_0$ and terminal payoff
$f(x(T))$, and the terminal state is interior to $X$,
then it also satisfies the transversality condition
$p(T)=\mathrm df(x(T))$.
\end{Thm}

\begin{Rem}
When trajectories meet $\partial X$, state constraints may
require a subdifferential formulation of the Hamiltonian
equations, with singular behavior of the momentum component;
see \cite{AiT89}.
\end{Rem}

\subsection{Symplectic geometry behind the Hamiltonian system}
Much later, the Hamiltonian system \eqref{HDE} was embedded in a very rich geometric context of symplectic geometry \cite{mmkm} which revealed hidden symmetries of the Hamiltonian systems. To recall only the basic tenets of the relevant construction: a symplectic manifold is a real even-dimensional manifold $M^{2n}$ equipped with an antisymmetric differential 2-form $\omega$ (the symplectic structure) which is closed, $\d \omega =0$ and nondegenerate (the map $\mathrm i\:T_x M\to T_x^*M$ defined below is an isomorphism between two real spaces $n$ $T_xM$ and $T^*_xM$). A vector field $Z$ on $M^{2n}$ is called \emph{Hamiltonian}, if the corresponding contraction is an exact 1-form, $\mathrm i_Z(\omega)=-\d H$, where $H\:M^{2n}\to\R$ is a real function on $M$ called the Hamiltonian of $Z$ which is denoted then usually as $X_H$. The most natural source of symplectic manifolds is the so called cotangent bundle $T^*X$ of a smooth $n$-dimensional manifold $X=X^n$ which consists of all pairs $(x,p)$, $x\in X$, $T^*_xX\owns p\:T_x X\to \R$ a dual vector (covector). Then $T^*X$ can be equipped by the structure of a $2n$-dimensional manifold carrying a 1-form $\alpha=p\,\d x$. Its exterior derivative $\omega=\d\alpha$ is a symplectic structure on $M=T^*X^n=M^{2n}$. One can immediately check that the system \eqref{HDE} is indeed a Hamiltonian vector field.  The Euler-Lagrange equations \eqref{CEL} are transformed using the Legendre transform into much more symmetric form \eqref{HDE}.

\par\end{Small}

\subsection{Convex/concave Hamiltonians. Saddle points and invariant manifolds}\label{sec:saddleH}
Now we add to the technical conditions listed in \secref{sec:technical} the crucial \emph{concavity condition} on the Lagrangian.
\begin{Def}\label{def:UC}
We say that $L$ is \emph{strictly jointly concave} if its full Hessian
$D^2_{(x,v)}L$ is negative definite. This is the convention used below,
and later for discrete Lagrangians; it is stronger than strict concavity
in the usual sense.
\end{Def}
Assume henceforth that $X\subset\R^n$ is compact and convex and that
$L$ satisfies the technical conditions of \secref{sec:technical} and
is strictly jointly concave. Then the Hamiltonian \eqref{Ham} is
strictly convex in $p$ and strictly concave in $x$.
The function $L(x,0)$ has a unique maximizer $x_*\in X$:
\begin{equation}\label{stat-max}
 x_*=\argmax_{x\in X}L(x,0).
\end{equation}
For the following local Hamiltonian analysis, we additionally assume
that $x_*\in\operatorname{int}X$, so that $L_x(x_*,0)=0$.
Put $p_*=-L_v(x_*,0)$.

\begin{Prop}[see \cite{Rockafellar}]
Under the above assumptions, the corresponding
Hamiltonian vector field possesses the following properties:
\begin{enumerate}
\item
The point $(x_*,p_*)$ is a singular point of the Hamiltonian vector
field \eqref{HDE} on $T^*X$.

\item
This point is hyperbolic: exactly $n$ eigenvalues of the corresponding
linearization matrix lie in the left half-plane
$\C_-=\{\Re\l<0\}$, while exactly $n$ others lie in the right
half-plane $\C_+=\{\Re\l>0\}$.

\item
The contracting invariant subspace $E_-$ projects isomorphically onto
the coordinate $x$-space, while the expanding invariant subspace $E_+$
projects isomorphically onto the coordinate $p$-space.
\end{enumerate}
\end{Prop}

\begin{Small}
\begin{proof}[Sketch of the proof]
For an arbitrary Lagrangian one has no special reasons to look for stationary solutions of the equations \eqref{CEL} yet under the present assumptions the solution $x(t)\equiv x_*$ with zero velocity $v(t)\equiv0$ satisfies $\partial_x L(x_*,0)=0$ as follows from \eqref{stat-max}. The stationary solution $(x_*,p_*)$ of \eqref{HDE} is the Legendre transform of this stationary solution, which can be verified by a direct inspection. 

The linearization $(2n\times2n)$-matrix of the Hamiltonian vector field \eqref{HDE} is equal to 
\begin{equation}
\begin{pmatrix}
\phantom{-}H_{px} & \phantom{-}H_{pp}
\\
-H_{xx} & -H_{xp}
\end{pmatrix}.
\label{Q-function}
\end{equation}
Consider the bilinear form which is a quadratic function
\begin{equation}\label{lyap}
V(x,p)=\<p-p_*,x-x_*\>\:T^*\R^n\to\R.
\end{equation}
Then the straightforward derivation of $V$ along the system linearized at $(x_*,p_*)$, yields 
\begin{equation}
\dot V=\<p-p_*,H_{pp}(p-p_*)\>-\<H_{xx}(x-x_*), x-x_*\>, 
\end{equation}
since the partial Hessians $H_{pp}>0$ and $H_{xx}<0$ are positive, resp., negative definite while the terms involving $H_{xp}=H_{px}^{T}$ cancel each other. In other words, the function $V$ is a Lyapunov function, monotonically growing along trajectories of \eqref{HDE}.

The hyperbolicity of the system \eqref{HDE} follows from the general fact that the eigenvalues of Hamiltonian singularities are symmetric with respect to the imaginary axis: together with each eigenvalue $\l\in\C$ it has an opposite eigenvalue $-\l\in\C$, see \cite{mmkm}. Thus it suffices to show that the imaginary axis $\mathrm i\R=\C\ssm(\C_-\cup\C_+)$ is free of such eigenvalues. This follows from the monotonicity of $V$. Indeed, a purely imaginary eigenvalue would yield a nonzero periodic real solution (or a nonzero stationary solution for the zero eigenvalue) of the linearized system, contradicting the strict increase of $V$ along every nonstationary solution. 

Denote by $E_\pm$ the invariant subspaces of the linearized system, corresponding to the partition of the eigenvalues between $\C_\pm$. Then the same monotonicity arguments prove that on the stable (contracting) subspace $E_-$ the function $V$ increases along any trajectory to its limit value $V(x_*,p_*)=0$ and hence must be strictly negative, so $E_-$  cannot intersect the coordinate plane $x=x_*$.  Similarly the expanding (unstable) space should be disjoint with the plane $p=p_*$.  
\end{proof}

\par\end{Small}

By the standard Stable Manifold Theorem (see, e.g., \cite{inv-manif}), the hyperbolic equilibrium (singular point of a vector field) $(x_*,p_*)$ 
possesses $n$-dimensional \emph{local} stable and unstable manifolds tangent at $(x_*,p_*)$ to the corresponding invariant subspaces. Denote them by $W_-$ and $W_+$ respectively: by construction, $W_-$ is a ``horizontal'' $n$-submanifold. Invariance means that trajectory passing through any point $x_0\in W_-$ remains forever on $W_-$ and tends to $(x_*,p_*)$ as $t\to+\infty$. A symmetric claim holds also for the other ``vertical'' invariant manifold $W_+$ which consists of orbits converging to the saddle point $(x_*,p_*)$ in the inverse time, $t\to-\infty$. Formally both $W_\pm$ should be understood as \emph{germs} of submanifolds. 

\begin{Prop}
Both (germs of) invariant manifolds $W_\pm$ are \emph{Lagrangian} submanifolds in the symplectic manifold $T^*X$, which means that the symplectic form $\omega=\d p\land \d x$ restricted on them vanishes identically.
\end{Prop}

\begin{proof}
The flow of any Hamiltonian vector field preserves the symplectic form \cite{mmkm}. On the other hand, since restriction of \eqref{HDE} is contracting (in the inverse time for $W_+$), in the limit $\omega$ must tend to zero. This is possible if and only if $\omega\big|_{W_\pm}\equiv0$. Indeed, for any $z\in W^-$ and $u,v\in T_zW^-$, the invariance of the
symplectic form under the Hamiltonian flow $\Phi^t$ implies that
\begin{equation}
\omega_z(u,v)
=
\omega_{\Phi^t(z)}\bigl(D\Phi^t(z)u,D\Phi^t(z)v\bigr).
\end{equation}
Since the restriction of the flow to $W^-$ is contracting,
$D\Phi^t(z)u,D\Phi^t(z)v\to0$ as $t\to+\infty$. Hence the right
hand side tends to zero, while the left hand side is independent of
$t$, and therefore $\omega_z(u,v)=0$. Thus $W^-$ is isotropic.
Since $\dim W^-=n=\frac12\dim T^*X$, it is Lagrangian. The same
argument applied to $W^+$ in the inverse time, $t\to-\infty$, proves
that $W^+$ is Lagrangian as well.
\end{proof}

\begin{Cor}[see \cite{mmkm}]\label{cor:Lagrangian}
The horizontal manifold $W_-$ locally can be represented as the derivative graph $p=\partial_x S_-(x)$ of a smooth function $S_-\:U\to \R$ on a neighborhood $U$ of $x_*$. Analogously, the vertical manifold can be represented as the derivative graph $x=\partial_p S_+(p)$ of a smooth function $S_+\:P\to\R$ on a neighborhood $P$ of $p_*$. \qed
\end{Cor}

\begin{Small}
\subsection{Bellman semigroups in continuous time: local symplectic interpretation}
For a terminal bonus $f$ the continuous-time Bellman operator is
formally defined by
\begin{equation}\label{c-bellman}
 (\B^T f)(a)=\sup_{x(\cdot)}\left\{
 \int_0^T L(x(t),\dot x(t))\,\d t+f(x(T))\right\},
 \qquad x(0)=a,
\end{equation}
where the supremum runs over admissible absolutely continuous paths.
Whenever the variational problems are well posed, dynamic programming
gives $\B^{T+S}=\B^T\B^S$. We do not claim that this semigroup
preserves any class of smooth terminal functions.

Suppose that $u(T,x)=(\B^T f)(x)$ is smooth in a region where the
state constraint is inactive. The dynamic programming principle then
gives the Hamilton--Jacobi equation
\begin{equation}\label{HJ}
 \partial_Tu(T,x)=H(x,\partial_xu(T,x)),\qquad u(0,x)=f(x).
\end{equation}
With our sign convention \eqref{HDE}, the graph
$p=\partial_xu(T,x)$ evolves under the \emph{inverse} Hamiltonian
flow, as long as the graph remains smooth.

The stable Lagrangian manifold constructed above gives a local
stationary solution. Indeed, its generating function $S_-$ satisfies
\begin{equation}\label{local-stationary-HJ}
 H(x,\partial_xS_-(x))=H(x_*,p_*)=L(x_*,0),\qquad x\in U,
\end{equation}
since $H$ is constant along Hamiltonian trajectories and all
trajectories in $W_-$ converge to $(x_*,p_*)$.
Thus $u(T,x)=S_-(x)+T L(x_*,0)$ is a local classical solution
of \eqref{HJ}. This is the local Hamilton--Jacobi counterpart of
the Bellman spectral equation. Identifying it with the value of a
global Bellman semigroup would require additional verification and
is not asserted here.

\begin{Rem}
The same geometric picture also describes long but finite
optimal trajectories. Under suitable concavity and regularity
assumptions, the classical turnpike theorems assert that such
trajectories spend most of their lifetime near the hyperbolic
stationary point, with possible deviations confined to
initial and terminal boundary layers. The underlying mechanism
is the strict monotonicity of the Lyapunov function $V$,
whose derivative vanishes only at the stationary point.
For details, see \cites{dan-87,AiT89}.
\end{Rem}

\par\end{Small}

\subsection{Lyapunov hyperbolicity}\label{sec:lyapunov}
In fact, the function 
\begin{equation}\label{V-single}
V(x,p)=\<p-p_*,x-x_*\>
\end{equation} 
introduced in \eqref{lyap} is the Lyapunov function not just for linearization of the Hamiltonian system \eqref{HDE}, but for the initial nonlinear system itself \eqref{HDE}.  But even this is a specific partial case of a more general structure, as we shall show in a moment. 

Consider the case where $X=\R^n$ is an affine subset and $M=T^*X=\R^n\times\R^{n*}$ is a linear symplectic $2n$-space equipped with the natural pairing $(x,p)\mapsto\<p,x\>$. Then any function $H\:M\to\R$ generates the Hamiltonian flow \eqref{HDE}. Moreover, we can consider non-autonomous Hamiltonians of the form $H=H(t,x,p)$ explicitly depending on time $t\in\R$. The equations will still have the same form, but the Hamiltonian ceases to be constant along orbits of the flow:
\begin{equation}
\tfrac{\d}{\d t} H (t,x,p)=\tfrac{\partial}{\partial t} H(t,x,p).
\end{equation}

Now we can take two identical copies $M_1\simeq M_2\simeq M$ and their $4n$-dimensional Cartesian product $M_1\times M_2$ with the coordinates $(x_1,p_1,x_2,p_2)$. Then any Hamiltonian flow with $H\:M\to\R$ can be diagonally lifted on this Cartesian square flow as follows:
\begin{equation}\label{H-square}
\left\{
\begin{aligned}
\dot x_i=&\phantom{-}\dfrac{\partial H(t,x_i,p_i)}{\partial p_i},
\\
\dot p_i=&-\dfrac{\partial H(t,x_i,p_i)}{\partial x_i}, 
\end{aligned} 
\right.
\qquad i=1,2.
\end{equation}
Trajectories of this Cartesian flow are simply (ordered) pairs of trajectories $(x_1(t),p_1(t))$, $(x_2(t),p_2(t))$ of the initial Hamiltonian flow \eqref{HDE} on $M$. 

\begin{Def}[Lyapunov structure]\label{def:lyap-str}
The quadratic form 
\begin{equation}\label{lyap-str}
V\:M_1\times M_2\to\R, \quad V(x_1,p_1,x_2,p_2)=\<p_1-p_2,x_1-x_2\>,
\end{equation}
will be called the \emph{Lyapunov structure} on the Cartesian square.
\end{Def} 
It is obviously symmetric with respect to the permutation of the indices $i=1,2$. Algebraically the form $V$ on $\R^{4n}$ has rank $2n$, null space (the diagonal $x_1=x_2$, $p_1=p_2$) also of dimension $2n$. After factorization by the diagonal the form becomes hyperbolic nondegenerate with the signature $(n,n)$. 

Assume that $H=H(t,x,p)$ is a Hamiltonian, concave in $x\in\R^n$ and convex in $p\in\R^{n*}$. The following argument does not require autonomy and remains valid for $H=H(t,x,p)$.
\begin{Thm}\label{thm:lyap+}
If the Hamiltonian $H\:\R\times\R^n\times\R^{n*}$ is convex/concave as above, then the Lyapunov form \eqref{lyap-str} is non-decreasing along any orbit of the diagonal flow \eqref{H-square}.
\end{Thm}

\begin{proof}
The proof is absolutely elementary. The full derivative of $V$ by virtue of \eqref{H-square} is equal to 
\begin{equation}
\frac{\d V}{\d t}=\<\dot p_1 -\dot p_2,x_1-x_2\>+\<p_1-p_2, \dot x_1-\dot x_2\>.
\end{equation}
Afterwards we can consider $t$ as a fixed and drop it from the notation for greater transparency. We can replace the derivatives $\dot x_i,\dot p_i$ by their values from the right hand side of the system \eqref{H-square}. Recall that the concavity of a differentiable function $f=f(z)$ gives $f(b)\le f(a)+\<\partial_z f(a),b-a\>$ for any $a,b\in\R^n$; for convex functions the inequality is reversed. 

Consider four points $(x_i,p_j)\in\R^n\times\R^{n*}$ and denote $h_{ij}=H(t,x_i,p_j)$ (they may depend on $t$). Then we have the following four inequalities,
\begin{equation*}
\begin{aligned}
h_{21}-h_{11}
&\le \<H_x(t,x_1,p_1),x_2-x_1\>,
&
h_{12}-h_{22}
&\le \<H_x(t,x_2,p_2),x_1-x_2\>,
\\
h_{12}-h_{11}
&\ge \<H_p(t,x_1,p_1),p_2-p_1\>,
&
h_{21}-h_{22}
&\ge \<H_p(t,x_2,p_2),p_1-p_2\>.
\end{aligned}
\end{equation*}
On the other hand, the formula for the full derivative after substitution can be transformed as follows,
\begin{equation}\label{Vdot-hij}
\begin{aligned}
\dot V
={}&
-\<H_x(t,x_1,p_1)-H_x(t,x_2,p_2),x_1-x_2\>
\\
&
+\<H_p(t,x_1,p_1)-H_p(t,x_2,p_2),p_1-p_2\>
\\
\ge{}&
-(h_{11}+h_{22}-h_{12}-h_{21})
+(h_{11}+h_{22}-h_{12}-h_{21})
=0.
\end{aligned}
\end{equation}
The mysterious cancellation is, of course, a consequence of the specific form of the Hamiltonian structure.
\end{proof}

\begin{Rem}
It should be stressed that the Lyapunov structure $V$ as it is defined by \eqref{lyap-str} on $\R^{2n}\times\R^{2n*}$ is a universal  algebraic object and does not depend on the choice of the Hamiltonian. This universal structure affects Hamiltonian dynamics in the saddle case. Its multiplicative counterpart $Q$ defined on $\R^{2n}_+\times\R^{2n*}_+$ has an analogous nature. Both are  moulds that can be used to construct the Lyapunov functions (functions monotonic along trajectories of flows) for special classes of Hamiltonial flows. This explains why the name of A. M. Lyapunov is affixed to these structures.  
\end{Rem}

\begin{Small}
\begin{Rem}
The above proof actually works under the sole assumption that $H$ is a saddle (convex/concave) Hamiltonian. No differentiability should be assumed, if by the partial derivatives $H_x,H_p$ we understand any elements of the corresponding subdifferentials/superdifferentials, see \cite{rtr}. 

Moreover, in the case where the Hamiltonian is allowed to take extended-real-values in $\R\cup\{\pm\infty\}$ while keeping the appropriate semicontinuity, the Hamiltonian system \eqref{HDE} ceases to be the necessary condition for optimality of the corresponding problem because of the so called phase constraints: the supporting trajectory $(x(t),p(t))$ will be not absolutely continuous anymore, and in the right hand side \eqref{HDE} one would have to add singular measures. Yet these measures have sign so that the monotonicity of the Lyapunov form persists, see \cite{func-an}. 

The above proof is transparent enough to determine explicitly the instances where the derivative $\dot V$ can vanish rather than remain positive: usually these instances are associated with the loci where $H$ is non-strictly convex/concave. 
\end{Rem}

\begin{Rem}
Theorem~\ref{thm:lyap+} actually explains the origin of the Lyapunov function \eqref{V-single}. If $(x_*,p_*)$ is a saddle point of the (stationary) saddle Hamiltonian $H(x,p)$, then the corresponding stationary solution of \eqref{HDE} can be plugged in the Lyapunov structure form, which yields precisely the Lyapunov function \eqref{V-single} for other (nonstationary) solutions of \eqref{HDE}.
\end{Rem} 
\par\end{Small}

\subsection{Lyapunov hyperbolicity: projective flavor}\label{sec:m-lyapunov}
\begin{Small}
The (additive) Lyapunov structure has an immediate multiplicative (projective) analog. Assume that in addition to the convexity/concavity assumptions, the Hamiltonian $H$ is defined only for $x\in\R^n_+$, $p\in\R^{n*}$ and is positive homogeneous of degree $1$ in each variable:
\begin{equation}
\forall \l,\mu\in\R_+\qquad H(t,\l x,\mu p)=\l\mu\cdot H(t, x,p).
\end{equation}
An important class of such Hamiltonians arises from Lagrangians which are indicator functions of closed convex technological cones $L=L_K(x,v)$,  $K\subseteq\R^n_+\times\R^n$, 
$$
L=L_K(x,v)=\begin{cases} 0,\qquad &(x,v)\in K, \\ -\infty &\text{ otherwise}.
\end{cases}
$$ 
Here $v$ is a continuous-time velocity, which need not be nonnegative. In economics such Lagrangians correspond to a special case of the so called Leontieff--von Neumann--Gale models of technological growth. The cone $K$ is called \emph{the technological cone}: in the discrete-time setting $K\subseteq\{(x,y)\in\R^n_+\times\R^n_+\}$ means that transition from the state $x$ to the state $y$ is technologically feasible, see \cite{mak-rub}. There are natural assumptions on $K$: convexity, closedness and absence of the horn of plenty: $(0,v)\notin K$ for any $v\ne 0$. To ensure that the admissible velocity set is nonempty at each nonzero state, we additionally assume that for every $x\in\R^n_+\setminus\{0\}$ there exists a finite $\lambda_x<0$ such that $(x,\lambda_x x)\in K$ (with $K$ here understood in state--velocity coordinates). The continuous-time analogue was discussed in detail in  \cites{dan-87,AiT89}.   

In this case the role of the quadratic form $V$, \eqref{lyap-str}, is played by a simpler looking quartic expression $Q$, called the \emph{multiplicative Lyapunov structure}. 
\begin{Thm}\label{thm:mlyap}
For the Cartesian flow associated with a saddle bihomogeneous Hamiltonian $H\:\R\times\R^n_+\times\R^{n*}_+$, the multiplicative Lyapunov structure, the quartic expression
\begin{equation}\label{V-star}
Q(x_1,p_1,x_2,p_2)=\<p_1,x_2\>\cdot\<p_2,x_1\>, \quad x_i\in\R^n_+,\ p_i\in\R^{n*}_+
\end{equation}
is monotonically non-increasing along solutions of \eqref{H-square}.
\end{Thm}

\begin{proof}
As before, after writing the derivatives we may regard $t$ as a fixed
parameter. Consider the four points $(x_i,p_j)$, $i,j=1,2$ and denote $h_{ij}=H(t,x_i,p_j)$ exactly as before.
By concavity of $H$ in $x$, convexity in $p$, and the Euler identities
\begin{equation}\label{Euler-bi}
 \<H_x(t,x,p),x\>=H(t,x,p)
 =
 \<H_p(t,x,p),p\>.
\end{equation}
coming from the bihomogeneity of $H$, we have
\begin{equation}\label{mult-four-ineq}
\begin{aligned}
h_{21}
&\le \<H_x(t,x_1,p_1),x_2\>,
&
h_{12}
&\le \<H_x(t,x_2,p_2),x_1\>,
\\
h_{21}
&\ge \<H_p(t,x_2,p_2),p_1\>,
&
h_{12}
&\ge \<H_p(t,x_1,p_1),p_2\>.
\end{aligned}
\end{equation}
Indeed, the first two inequalities follow from concavity and
positive homogeneity in $x$, while the last two follow from convexity
and positive homogeneity in $p$.

Consider now the two cross-pairings occurring in the numerator.
Using \eqref{H-square} and \eqref{mult-four-ineq}, we obtain
\begin{equation}\label{cross-pairings}
\begin{aligned}
\frac{\d}{\d t}\<p_1,x_2\>
&=
-\<H_x(t,x_1,p_1),x_2\>
+\<p_1,H_p(t,x_2,p_2)\>
\\
&\le -h_{21}+h_{21}=0,
\\[1mm]
\frac{\d}{\d t}\<p_2,x_1\>
&=
-\<H_x(t,x_2,p_2),x_1\>
+\<p_2,H_p(t,x_1,p_1)\>
\\
&\le -h_{12}+h_{12}=0.
\end{aligned}
\end{equation}
Thus, each of the two cross-pairings is non-increasing along the
Cartesian flow. Since on $\R^n_+\times\R^{n*}_+$, the pairings are non-negative, hence
their product $Q=\<p_1,x_2\>\<p_2,x_1\>$ 
is non-increasing as well. 
\end{proof}

\begin{Cor}
The rational function 
\begin{equation}\label{Qstar}
Q^\star(x_1,p_1,x_2,p_2)=
     \frac{\<p_1,x_2\>\<p_2,x_1\>}
     {\<p_1,x_1\>\<p_2,x_2\>}.
\end{equation}
is monotonically non-increasing for solutions \eqref{H-square} lying in the domain of this function, i.e. in the invariant domain $\<p_i,x_i\>\ne 0$.
\end{Cor}
\begin{proof}
The Euler identities \eqref{Euler-bi} imply that each of the terms in the denominator of $Q^\star$ is constant along \eqref{H-square}. 
\end{proof}

\begin{Rem}
The rational expression for $Q^\star$ is zero-homogeneous on each variable thus $Q^\star$ descends under independent positive rescalings of its four arguments to an open subset of a compact product of four projective simplices. We reiterate here that the Lyapunov structures $V$ and $Q$ are universal and do not depend on the choice of Hamiltonians. 
\end{Rem}

\begin{Rem}
In the same way as the additive Lyapunov structure in the stationary (autonomous) case generates a genuine Lyapunov function \eqref{V-single} if $(x_*,p_*)$ is a stationary solution (the saddle point of the Hamiltonian), the multiplicative Lyapunov structure $Q$ produces a genuine Lyapunov function for \eqref{HDE}. Suppose that in the autonomous bihomogeneous case the system possesses an exponential solution of the form
\begin{equation}\label{gale}
x(t)=\exp \l t\cdot x_*,\ p(t)=\exp (-\l t)\cdot  p_*.
\end{equation} 

The ray $\mathbb R_+x_*$, corresponding to what in the von Neumann--Gale terminology is a stationary (von Neumann) growth process, together with the growth exponent $\lambda$, is obtained from the associated static growth problem 
\begin{equation}\label{vN-ray}
\sup_{\l\in \R,\ x_*\in\R^n_+}\{\l: x_*\ne 0,\ (x_*,\l x_*)\in K\},
\end{equation}
(existence and attainability of the maximum requires certain economically meaningful assumption in the technological cone $K$, like closedness and absence of the horn of plenty). 
Then the function $\<p_*,x\>\<p,x_*\>$ will be a (non-increasing) Lyapunov function for \eqref{HDE}. 
\end{Rem} 
\par\end{Small}

\subsection{The goalposts}
The goal of this crash course in the classical calculus of variations was to show that the continuous time formulation of the stationary dynamic optimization problem allows a very geometric treatment using symplectic geometry and Hamiltonian dynamics. On the other hand, technically the discrete time formulation allows us to avoid many analytic difficulties. It would be great if in fact the two theories could be married into a single one, at least in the strictly jointly concave case. To do this, we would have to solve the interpolation problem, also known as the embedding problem.
We want to study this problem in different contexts, varying from a finite (discrete) set $X$ and all the way to the strongly concave operators on a compact subspace $X\Subset\R^n$. Before addressing the embedding problem, we first need to develop the discrete Hamiltonian counterpart of the continuous picture.

\section{Hamiltonian pseudocascades}
\label{sec:hamiltonian-cascades}

The discussion in the preceding section suggests why the Hamiltonian
formalism is so useful for variational problems on long and infinite
intervals.  In continuous time, passing from the Euler--Lagrange
second order equation to the Hamiltonian first order system reveals a
symplectic phase space, saddle points and their invariant manifolds.
Infinite extremals are naturally associated with the separatrices of
such a saddle, while extremals on a long finite interval shadow a
concatenation of an incoming and an outgoing separatrix: the former is
determined by the initial state, whereas the latter is selected by the
terminal term, transported backwards by the Bellman evolution.

The purpose of this section is to understand which part of this
picture survives in discrete time.  There is more than one answer to
the apparently innocent question of what a ``discrete Hamiltonian
system'' should mean.  We first recall the phenomenological answer of
Cass and Shell, then the variational answer of Marsden and West, and
finally formulate the intrinsic\emph{ analytic }construction which will be used in the
rest of the paper.  Most ingredients of the latter belong to the
standard theory of discrete mechanics.

\begin{Small}
\subsection{The economists' approach}\label{sec:economics}

The Hamiltonian approach to economic dynamics enjoyed a
remarkable period of popularity in the 1970s. A particularly
representative landmark was the 1976 special issue of the
Journal of Economic Theory, subsequently published as
\emph{The Hamiltonian Approach to Dynamic Economics},
edited by Cass and Shell \cites{cass-shell,cass-shell-intro}.
The attraction was not merely formal. The theory of
economic growth was seeking a unified treatment of
multisector models, and the Hamiltonian systems studied
by Rockafellar offered a powerful tool: a Lyapunov
function associated with the saddle geometry of
convex--concave Hamiltonians. Its monotonicity provided
a route to the stability and turnpike properties of
optimal economic trajectories.

Extending this machinery to discrete time proved less
straightforward. Cass and Shell approached the problem
through Hamiltonian functions representing economic
technologies and postulated difference equations
motivated by investment and capital rentals. These
relations retained enough of the saddle structure to
support useful Lyapunov arguments, but did not reproduce
the full geometry of continuous Hamiltonian dynamics.
In particular, even an autonomous Hamiltonian need not
be conserved by the resulting discrete evolution.
Viewed retrospectively, this was a revealing compromise:
the desired stability mechanism survived, while the
discrete equations were not obtained from a discrete
variational principle. The latter distinction will be
essential for our purposes.

\subsection{The variational approach of Marsden--West}

A substantially different perspective emerged from
discrete mechanics and the theory of variational
integrators, systematically developed and surveyed
by Marsden and West \cite{marsden-west}. The original
problem was numerical: how to discretize continuous
mechanical systems while preserving their geometric
structure. The decisive change of viewpoint was to
discretize the variational principle itself, rather
than the differential equations derived from it.
The resulting discrete Euler--Lagrange equations
inherit a symplectic structure through the discrete
Legendre transformations, without requiring a
distinguished Hamiltonian function generating
their evolution.

For our purposes, the importance of this construction
lies beyond its numerical origins. Once the discrete
variational principle is taken as the starting point,
there is no longer any need for a continuous-time
ancestor. An arbitrary regular discrete Lagrangian
gives rise to the same symplectic geometry, at least
locally. Thus a method designed to approximate
continuous dynamics also provides an intrinsic
geometric language for discrete variational problems.
This is the aspect of the Marsden--West framework
that we shall exploit below.

The original numerical motivation nevertheless
provides a useful consistency check. Whenever the
discrete Lagrangian is the exact action of a continuous
variational problem over a time interval of length
$h$, the corresponding discrete Hamiltonian evolution
coincides, with the appropriate normalization of
momenta, with the time-$h$ map of the continuous
Hamiltonian flow, wherever the relevant local maps
are defined. Thus the construction recovers the
continuous theory exactly when a continuous ancestor
is present, while remaining meaningful when there
is none.
\par\end{Small}

\subsection{Known caveats}

Transition from continuous to discrete time should not be expected to be a smooth ride (pun intended). One source of difficulties is the tension between compactness of the state space $X$ in stationary infinite-horizon optimization and the absence of active state constraints. In continuous time, the classical Euler--Lagrange differential equations describe extremals as long as the relevant variations are unconstrained. When an extremal $x(\cdot)$ reaches the boundary $\partial X$, the state constraint may destroy the absolute continuity of the supporting momentum trajectory $p(t)$, requiring a more general form of the Hamiltonian equations \eqref{HDE}. Remarkably, the resulting singular contributions may even reinforce Lyapunov hyperbolicity, as shown in \cite{func-an}.

In discrete time, state constraints necessarily manifest themselves differently. The transition from $t$ to $t+1$ involves a finite step rather than infinitesimal evolution. Consequently, local solvability of the discrete Euler--Lagrange equations does not guarantee a single-valued evolution on the entire admissible phase domain.

Ideally, the passage to discrete time would replace the flow of a smooth Hamiltonian vector field \eqref{HDE} by iterations of a smooth map $\Phi$, thus producing a discrete dynamical system, or a \emph{cascade}. However, the discrete variational equations generally define only a multivalued evolution, naturally represented by a symplectic correspondence. Even after selecting a smooth local branch, the resulting map need not preserve its domain, so its iterates need not be defined indefinitely. We shall call such locally invertible correspondences \emph{pseudocascades}; their individual local branches are ordinary maps. The prefix emphasizes the absence of a globally defined iterated dynamics, rather than any failure of the symplectic identities.

\subsection{Euler--Lagrange and Hamiltonian pseudocascades}

We therefore return to an arbitrary smooth discrete Lagrangian
$b\:X\times X\to\R$, not assumed to have a continuous-time ancestor in the sense of Marsden--West.
The interior extremals of the corresponding stationary dynamic optimization problem satisfy the discrete Euler--Lagrange equation.
\begin{equation}\label{del}
\partial_{y} b(x_{t-1},x_t)+\partial_{x} b(x_t,x_{t+1})=0.
\end{equation}

To write the (second order nonlinear) difference equation for extremals $\{x_t\}_{t=0}^{T}$ (finite or infinite for $T=+\infty$), this equation has to be solved in the form defining forward evolution, i.e., for any $t=1,2,\dots$, the value $x_{t+1}$ should be expressed as the function of $x_{t-1},x_t$. The value $x_0$ is specified as the initial condition of the extremal, thus the equation \eqref{del} would define a deterministic trajectory if the value $x_1$ is determined\footnote{In continuous time classical mechanics, simultaneous knowledge of $x_0$ and $x_1$ corresponds to specifying the initial positions and velocities of all particles of the system.}.  
The local solvability of \eqref{del} is governed by the mixed
derivatives of the discrete Lagrangian.

\begin{Def}\label{def:local-legendre}
The Lagrangian $b\:X\times X\to\R$ satisfies the \emph{local
(discrete) Legendre condition} if
\begin{equation}\label{local-legendre}
\det\partial^2_{xy}b(x,y)=\det\partial^2_{yx}b(x,y)\ne0
\end{equation}
at every point $(x,y)\in X\times X$.
\end{Def}

By the implicit function theorem, this condition makes
\eqref{del} \emph{locally} solvable in either time direction. 

\begin{Def}\label{def:Legendre-maps}
The two maps $\mathscr L^{\pm}:X\times X\to X\times\R^{n*}$, 
\begin{equation}\label{fiber}
\mathscr L^-(x,y)=\bigl(x,\partial_xb(x,y)\bigr),
\qquad
\mathscr L^+(x,y)=\bigl(y,-\partial_yb(x,y)\bigr),
\end{equation}
are called the \emph{Legendre maps} (or $b$-Legendre maps when the Lagrangian $b$ needs to be specified).
\end{Def}

Assuming local solvability of \eqref{del}, we introduce the (possibly multivalued and partially defined) Euler--Lagrange correspondence $\Xi$ by
\begin{equation}\label{EL-cascade}
\Xi(x,y)=\bigl\{(y,z)\in X\times X:
\partial_y b(x,y)+\partial_x b(y,z)=0\bigr\}.
\end{equation}
Thus $(y,z)\in\Xi(x,y)$ if and only if
\begin{equation}\label{matching-map}
\mathscr L^+(x,y)=\mathscr L^-(y,z).
\end{equation}

\begin{Def}\label{def:pseudocascade}
A \emph{pseudocascade} on a smooth manifold $M$ is a smooth
correspondence $\mathscr C\subseteq M\times M$ such that both
coordinate projections
\begin{equation}
\pi_1,\pi_2:\mathscr C\to M
\end{equation}
are local diffeomorphisms.

A trajectory of $\mathscr C$ is a finite or infinite sequence
$\{z_t\}$ satisfying $(z_t,z_{t+1})\in\mathscr C$ for every
consecutive pair of indices.
\end{Def}

\begin{Small}
\begin{Rem}
The above definition assumes that $M$ is a smooth manifold.
We shall use the same term for correspondences on manifolds
with boundary whose restrictions to the interior are
pseudocascades and which extend to the boundary as closed
correspondences.
\end{Rem}
\par\end{Small}

On every open region where the local Legendre condition holds, Definition \ref{def:pseudocascade} implies that the
Euler--Lagrange correspondence $\Xi$ in \eqref{EL-cascade}
restricts to a pseudocascade, which we shall call
the \emph{Lagrangian pseudocascade}.

The Legendre maps \eqref{fiber} allow us to introduce
the corresponding correspondence on the symplectic
phase space $X\times\R^{n*}$:
\begin{equation}\label{SC}
\Theta=\bigl\{(\mathscr L^-(x,y),\mathscr L^+(x,y)):\ (x,y)\in X\times X\bigr\}.
\end{equation}
Equivalently, $\Theta=\mathscr L^+\circ(\mathscr L^-)^{-1}$ as a correspondence.
The local Legendre condition guarantees that $\Theta$
is locally represented by symplectic diffeomorphisms, denoted by $\Phi$.
Likewise, local branches of $\Xi$ will be denoted by $\Psi$. Using these notations is convenient provided that every time talking about a specific $\Phi$ or $\Psi$, the domain and the choice of branch should be clear from the context.

Along any Euler--Lagrange trajectory, the local branches $\Psi$ and $\Phi$ can be chosen compatibly, so that
\begin{equation}\label{ConjPhiPsi}
\Phi\circ\mathscr L^-=\mathscr L^-\circ\Psi,
\qquad
\mathscr L^+\circ\Psi=\Phi\circ\mathscr L^+
\end{equation}
hold wherever the indicated compositions are defined along that trajectory. These are not identities between globally defined maps or unrestricted compositions of correspondences. Both identities are understood locally, on suitably chosen and mutually compatible domains of definition of the respective branches.

\begin{Def}\label{def:Ham-cascade}
The correspondence $\Theta$ defined by
\eqref{SC} is called the \emph{Hamiltonian correspondence}
generated by $b$. Its restriction to any regular region where the local Legendre condition holds is the \emph{Hamiltonian pseudocascade} generated by $b$.
\end{Def}

The basic geometric fact is standard in discrete mechanics, but its
proof is short enough to be more informative than a reference.

\begin{Def}\label{def:symplectic-pseudocascade}
Let $(M,\omega)$ be a symplectic manifold.
A pseudocascade $\mathscr C\subseteq M\times M$ is called
\emph{symplectic} if
\begin{equation}\label{symplectic-correspondence}
\left.(\pi_1^*\omega-\pi_2^*\omega)\right|_{\mathscr C}=0,
\end{equation}
where $\pi_1,\pi_2:M\times M\to M$ are the coordinate
projections.
\end{Def}

\begin{Thm}\label{thm:symplectic-cascade}
On any region where the local Legendre condition holds, the Lagrangian pseudocascade $\Xi$ is symplectic (on each local branch $\Psi$) with respect to a
$b$-dependent symplectic form $\varOmega$ on $X\times X$, 
\begin{equation}\label{Omega}
\varOmega
=\sum_{i,j=1}^n(\partial^2_{xy}b)_{ij}\,\d y_j\land\d x_i.
\end{equation}

The
Hamiltonian pseudocascade $\Theta$ is symplectic (on each local branch $\Phi$) with respect to the
canonical form $\omega=\d p\land\d x$ on $M=X\times\R^{n*}$.
\end{Thm}

\begin{proof}
Let $\alpha=p\,\d x$, so that $\omega=\d\alpha$.  By
\eqref{fiber},
\begin{equation*}
(\mathscr L^-)^*\alpha=\partial_{x} b\,\d x,
\qquad
(\mathscr L^+)^*\alpha=-\partial_{y} b\,\d y.
\end{equation*}
Their difference is $\d b$.  Hence
$(\mathscr L^-)^*\omega=(\mathscr L^+)^*\omega$ which coincides with $\varOmega$. 
The nondegeneracy of $\partial^2_{xy}b$ makes $\varOmega$ symplectic.  On compatible local branches,
$\mathscr L^+=\mathscr L^-\circ\Psi$ and
$\mathscr L^+=\Phi\circ\mathscr L^-$ yield
$\Psi^*\varOmega=\varOmega$ and $\Phi^*\omega=\omega$.
\end{proof}

Thus the passage from $\Xi$ to $\Theta$ does not create new dynamics;
it places the same dynamics on a phase space carrying the universal
canonical symplectic form.  This is the representation used below.

\subsection{Lagrangian submanifolds and discrete Hamilton--Jacobi equations}

Recall that an $n$-dimensional submanifold $W$ of the symplectic
phase space $(M,\omega)$ is called \emph{Lagrangian} if
$\omega\big|_W=0$. Locally, when $M=X\times\R^{n*}$ and the canonical projection
$\pi_x\:W\to X$ is nonsingular (we call then $W$ horizontal), such a manifold is the graph of a
closed $1$-form and hence, locally,
$
W=\{(x,\d S(x))\}
$
for a generating function $S$.

This elementary observation provides the geometric bridge from
symplectic pseudocascades to functional equations.

\begin{Prop}\label{prop:discrete-HJ}
Let $W=\operatorname{graph}\d S$ be a Lagrangian graph in the
symplectic phase space. Suppose that a smooth local branch $\Phi$
of $\Theta$ maps an open portion of $W$ into $W$.
Let $\Lambda_S$ be the corresponding smooth family of transitions
$(x,y)$, parameterized by their initial states. Then
$\partial_x b(x,y)=\d S(x)$ and
$\partial_y b(x,y)=-\d S(y)$ on $\Lambda_S$.
Consequently,
\begin{equation}\label{dhj}
b(x,y)+S(y)-S(x)=\operatorname{const},
\qquad (x,y)\in\Lambda_S,
\end{equation}
on each connected component of $\Lambda_S$.
\end{Prop}

Thus a Lagrangian graph $W=\operatorname{graph}\d S$ invariant under a local branch $\Phi$ of $\Theta$ gives rise naturally to a functional equation for its generating function $S$. Equations of this kind are commonly referred to as \emph{discrete Hamilton--Jacobi equations}. Their relation to discrete symplectic dynamics has been explored from several different viewpoints. Elnatanov and Schiff 
\cite{elnatanov-schiff} derived the discrete Hamilton--Jacobi
equation from canonical discrete dynamics; Lall and West
\cite{lall-west} formulated discrete Hamiltonian mechanics in terms
of generating functions; Ohsawa--Bloch--Leok
\cite{ohsawa-bloch-leok} developed the geometric discrete
Hamilton--Jacobi theory in terms of exact Lagrangian graphs; their linear theory explicitly identifies transversality to the vertical fibers as the condition under which a Lagrangian subspace is represented as $\operatorname{graph}\d S$. In the optimal-control setting,
Navasca \cite{navasca} obtained a local stable Lagrangian manifold
whose generating function is the optimal cost.

Notice that \eqref{dhj} is, at this stage, a geometric invariance equation on the Lagrangian correspondence $\Lambda_S$; it is not yet the Bellman equation, which involves maximization over all admissible transitions. The two will meet below under the concavity assumptions which select the invariant graph by optimization.

The above holds for any invariant Lagrangian manifold; moreover, there is no reason for such a manifold to exist globally. For the concave discrete Lagrangians considered in the next section, the Lyapunov and hyperbolic structure will provide precisely such a selection mechanism.

\section{Hamiltonian pseudocascades for concave Lagrangians}\label{sec:hyperbolic}
The preceding section introduced Hamiltonian pseudocascades as a geometric description of discrete variational dynamics. Motivated by the Lyapunov hyperbolicity of continuous-time Hamiltonian systems discussed in \secref{sec:lyapunov}, we now examine the structure imposed on these pseudocascades by concavity of the discrete Lagrangian. 

\begin{Def}
Throughout this section, strict concavity of a smooth Lagrangian means negative definiteness of its full Hessian. 
\end{Def}

\subsection{Lyapunov structure for Hamiltonian pseudocascades}
\label{sec:discrete-lyapunov}

For two phase points $(x_i,p_i)$, $i=1,2$, put
\begin{equation}\label{V-discrete}
V\big((x_1,p_1),(x_2,p_2)\big)
=\<p_1-p_2,x_1-x_2\>.
\end{equation}
In continuous time, the monotonicity of this function
along pairs of trajectories of saddle Hamiltonian systems
was the cornerstone of Lyapunov hyperbolicity.
Remarkably, its discrete counterpart follows directly
from the joint concavity of the Lagrangian, without
requiring a globally defined Hamiltonian cascade.

\begin{Thm}\label{thm:discrete-lyapunov}
Let $b\:X\times X\to\R$ be a smooth jointly concave discrete Lagrangian,
and let $\Theta$ be the Hamiltonian correspondence
generated by $b$ as in \eqref{SC}. Then, for any two trajectories
$(x_t^i,p_t^i)$, $i=1,2$, of $\Theta$,
\begin{equation}\label{V-monotone-discrete}
V(t+1)\ge V(t).
\end{equation}
If $b$ is strictly jointly concave, the inequality is
strict whenever the two corresponding transitions
$(x_t^1,x_{t+1}^1)$ and $(x_t^2,x_{t+1}^2)$
are distinct.
\end{Thm}

\begin{proof}
Write $(x_i,y_i)=(x_t^i,x_{t+1}^i)$ and
$(p_i,q_i)=(p_t^i,p_{t+1}^i)$.
By the Legendre relations,
$p_i=\partial_x b(x_i,y_i)$ and
$q_i=-\partial_y b(x_i,y_i)$.

The supporting-plane inequalities for the concave
function $b$, applied at $(x_1,y_1)$ and $(x_2,y_2)$,
give
\begin{equation}
\begin{aligned}
b(x_1,y_1)-b(x_2,y_2)
&\le
\<\partial_x b(x_2,y_2),x_1-x_2\>
+
\<\partial_y b(x_2,y_2),y_1-y_2\>,\\
b(x_2,y_2)-b(x_1,y_1)
&\le
\<\partial_x b(x_1,y_1),x_2-x_1\>
+
\<\partial_y b(x_1,y_1),y_2-y_1\>.
\end{aligned}
\end{equation}
Adding these inequalities and substituting the
Legendre relations yields
\begin{equation}
\<p_1-p_2,x_1-x_2\>
\le
\<q_1-q_2,y_1-y_2\>,
\end{equation}
which is precisely \eqref{V-monotone-discrete}.
Strict joint concavity gives strict inequality
for distinct transitions.
\end{proof}

\begin{Lem}[Asymptotic Lyapunov inequalities]
\label{lem:asymptotic-Lyapunov}
Let $b$ be a smooth jointly concave discrete Lagrangian generating a Hamiltonian pseudocascade $\Theta$ on the region containing the trajectories under consideration.
\begin{enumerate}
\item If two forward trajectories $z_t^1,z_t^2$ of $\Theta$
converge to the same phase point $z_*$, then
$V(z_0^1,z_0^2)\le0$. The inequality is strict if
$b$ is strictly jointly concave and the trajectories
are distinct.
\item Let $z_t$ be a forward trajectory converging to a
stationary point $z_*$, and let $(\xi_t,\eta_t)$ be a
nonzero tangent trajectory along it converging to zero.
If $D^2b$ is negative definite near $(x_*,x_*)$, then
$\<\eta_0,\xi_0\><0$.
\end{enumerate}
\end{Lem}

\begin{proof}
For two trajectories, the Lyapunov function
$V(z_t^1,z_t^2)$ is nondecreasing by
Theorem~\ref{thm:discrete-lyapunov} and tends
to zero as both trajectories converge to
the same phase point. Strict concavity
makes the one-step inequality strict
whenever the corresponding transitions differ.

For the tangent trajectory $(\xi_t,\eta_t)$, write
$\xi_{t+1}$ for its state component at the next step.
Differentiating the Legendre relations along the trajectory gives
\begin{equation}
\begin{aligned}
\eta_t&=b_{xx}(x_t,x_{t+1})\xi_t
          +b_{xy}(x_t,x_{t+1})\xi_{t+1},\\
\eta_{t+1}&=-b_{yx}(x_t,x_{t+1})\xi_t
          -b_{yy}(x_t,x_{t+1})\xi_{t+1}.
\end{aligned}
\end{equation}
Consequently, for $Q(\xi,\eta)=\<\eta,\xi\>$,
\begin{equation}
Q(\xi_{t+1},\eta_{t+1})-Q(\xi_t,\eta_t)
=-\begin{pmatrix}\xi_t\\\xi_{t+1}\end{pmatrix}^{\!*}
D^2b(x_t,x_{t+1})
\begin{pmatrix}\xi_t\\\xi_{t+1}\end{pmatrix}\ge0.
\end{equation}
Since $(\xi_t,\eta_t)\to0$, the nondecreasing sequence
$Q(\xi_t,\eta_t)$ tends to zero, so $Q(\xi_0,\eta_0)\le0$.
If $D^2b$ is negative definite near $(x_*,x_*)$, the
one-step inequality is strict for all sufficiently large $t$
for which the tangent trajectory is nonzero. Since both projections of the pseudocascade are local diffeomorphisms along the entire trajectory, a nonzero tangent trajectory cannot vanish at any finite time. Hence $Q(\xi_0,\eta_0)<0$.
\end{proof}

\subsection{Stationary solutions and linearization}
\label{sec:special-cascades}

We begin with stationary trajectories of the Hamiltonian correspondence $\Theta$. Let $b$ be a smooth jointly concave discrete Lagrangian on $X\times X$, and assume that the restriction $x\mapsto b(x,x)$ attains its maximum at an interior point $x_*\in\operatorname{int}X$. Then $\partial_xb(x_*,x_*)+\partial_yb(x_*,x_*)=0$, so $z_*=(x_*,\partial_xb(x_*,x_*))$ defines a stationary trajectory of $\Theta$. Under the local discrete Legendre condition, $z_*$ is a fixed point of the corresponding local symplectic branch $\Phi$.

Fixing the stationary trajectory $z_*$ as one of the arguments of the pairwise Lyapunov function \eqref{V-discrete}, we obtain an ordinary Lyapunov function
\begin{equation}\label{stationary-Lyapunov}
V_*(x,p)=\<p-p_*,x-x_*\>.
\end{equation}
Its quadratic part at $z_*$ is the form $Q(\xi,\eta)=\<\eta,\xi\>$ appearing below. Thus the linearized dynamics inherits the Lyapunov structure established for pairs of trajectories.

The following result is pure linear algebra.

\begin{Thm}\label{thm:discrete-hyperbolicity}
Let $b$ be a smooth jointly concave discrete Lagrangian on $X\times X$, and assume that the function $x\mapsto b(x,x)$ attains its maximum at $x_*\in\operatorname{int}X$. Suppose that the second derivatives of $b$ at $(x_*,x_*)$ satisfy
\begin{equation}\label{hyperbolic-jet}
D^2b(x_*,x_*)<0,
\qquad
\det\partial^2_{xy}b(x_*,x_*)\ne0.
\end{equation}
Put $p_*=\partial_x b(x_*,x_*)$ and denote $z_*=(x_*,p_*)$.

Then the Hamiltonian correspondence $\Theta$ generated by $b$ is a pseudocascade near the stationary transition and admits a unique local symplectic branch $\Phi$ through the stationary transition $(z_*,z_*)$. The point $z_*$ is a hyperbolic fixed point of $\Phi$, and its stable and unstable subspaces $E^-$ and $E^+$ both have dimension $n$.

Moreover, the quadratic form
\begin{equation}\label{Q-linear}
Q(\xi,\eta)=\<\eta,\xi\>
\end{equation}
is negative definite on $E^-$ and positive definite on $E^+$.
\end{Thm}
\begin{proof}
Put $T=D\Phi(z_*)$ and
$A=\partial_{xx}^2b(x_*,x_*)$,
$B=\partial_{xy}^2b(x_*,x_*)$,
$C=\partial_{yy}^2b(x_*,x_*)$.
The linearized Euler--Lagrange equation is
\begin{equation}\label{linear-DEL}
B^{\mathsf T}\xi_{t-1}
+(A+C)\xi_t+B\xi_{t+1}=0.
\end{equation}
If $\xi_t=\lambda^t v$ with $|\lambda|=1$, then
\begin{equation}
\bigl(\lambda B+A+C+
\lambda^{-1}B^{\mathsf T}\bigr)v=0.
\end{equation}
Multiplying by $v^*$ gives
\begin{equation}
\begin{pmatrix}v\\ \lambda v\end{pmatrix}^{\!*}
\begin{pmatrix}A&B\\ B^{\mathsf T}&C\end{pmatrix}
\begin{pmatrix}v\\ \lambda v\end{pmatrix}=0,
\end{equation}
contradicting negative definiteness of $D^2b(x_*,x_*)$.
Thus $T$ has no eigenvalues on the unit circle.
Since $T$ is symplectic, its stable and unstable
subspaces both have dimension $n$.

For $T(\xi,\eta)=(\xi',\eta')$, the linearized
Legendre relations read
\begin{equation}
\eta=A\xi+B\xi',
\qquad
\eta'=-B^{\mathsf T}\xi-C\xi'.
\end{equation}
Since $B$ is invertible, the first relation gives
$\xi'=B^{-1}(\eta-A\xi)$, so $(\xi,\xi')$
provide equivalent coordinates for $(\xi,\eta)$.
Consequently,
\begin{equation}
Q(\xi',\eta')-Q(\xi,\eta)
=
-\begin{pmatrix}\xi\\ \xi'\end{pmatrix}^{\!*}
D^2b(x_*,x_*)
\begin{pmatrix}\xi\\ \xi'\end{pmatrix}>0
\end{equation}
for every nonzero $(\xi,\eta)$.

The negative definiteness of $Q$ on $E^-$ follows
from Lemma~\ref{lem:asymptotic-Lyapunov}, applied
to the stable tangent trajectories at $z_*$. For every
nonzero $z\in E^+$, the sequence $Q(T^{-k}z)$ is
strictly decreasing and converges to zero. Hence
$Q(z)>0$ on $E^+$.
\end{proof}

\begin{Small}
\begin{Rem}
The bi-homogeneous case admits a parallel treatment,
mutatis mutandis, within the framework of nonsmooth
Hamiltonian correspondences. The primal--dual relations
for a technological cone replace the discrete Legendre
transformations, while the von Neumann rays established
by Makarov and Rubinov \cite{mak-rub} play the role of
exponential trajectories. The corresponding multiplicative
Lyapunov structure was described in
\secref{sec:m-lyapunov}.
\end{Rem}
\par\end{Small}

\subsection{Invariant local Lagrangian manifolds}
\label{sec:invariant-manifolds}

In continuous time variational optimization, the Hamilton--Jacobi equation (or, in its optimal-control formulation, the Hamilton--Jacobi--Bellman equation), together with the associated Lax--Oleinik semigroup \cites{Fathi1998,Fathi2008}, provides the functional description of the dynamics. Its geometric counterpart is the Hamiltonian flow: smooth solutions of the stationary Hamilton--Jacobi equation $H(x,\d S)=\lambda$ generate locally invariant Lagrangian graphs, and under suitable concavity assumptions the distinguished stationary solution generates the stable manifold of a hyperbolic equilibrium.
The continuous-time discussion in \secref{sec:variations} constructs the local stationary Hamilton--Jacobi graph; it does not identify this graph with a global Bellman value function.

In discrete time, much of the familiar Hamiltonian apparatus disappears, most notably the Hamiltonian function as an infinitesimal generator. What survives is the Bellman spectral equation on the functional side and the Hamiltonian pseudocascade on the symplectic side. We show below that their duality can be reconstructed at the level of germs near the stationary point, without imposing global conditions on the discrete dynamics. Extending this local picture beyond a neighborhood of the stationary point is a separate problem, inevitably involving state constraints and the behavior of the Lagrangian near $\partial X$. Such de-localization may require additional assumptions whose compatibility with the usual concavity and regularity hypotheses of economic models deserves further investigation.

\begin{Thm}\label{thm:local-stable-Lagrangian}
Under the assumptions of
Theorem~\ref{thm:discrete-hyperbolicity},
the Hamiltonian correspondence $\Theta$ admits
a unique germ $W^-$ of a smooth stable invariant
Lagrangian manifold through $z_*=(x_*,p_*)$,
associated with the local branch $\Phi$ fixing $z_*$.

The germ is horizontal and admits a representation
\begin{equation}\label{local-stable-graph}
W^-=\operatorname{graph}\d S,
\end{equation}
where $S$ is a smooth strictly concave function
defined in a neighborhood of the point $x_*$.
\end{Thm}

\begin{proof}
By Theorem~\ref{thm:discrete-hyperbolicity},
$z_*$ is a hyperbolic fixed point of the local
symplectomorphism $\Phi$, with an $n$-dimensional
stable subspace $E^-$. The classical Stable Manifold
Theorem \cite{inv-manif} yields the unique germ
of its smooth stable manifold $W^-$.

Since $\Phi$ is symplectic and tangent vectors
to $W^-$ contract under forward iteration,
the symplectic form vanishes on $W^-$. As
$\dim W^-=n$, the manifold is Lagrangian.

By Theorem~\ref{thm:discrete-hyperbolicity},
the quadratic form $Q(\xi,\eta)=\<\eta,\xi\>$
is negative definite on $E^-$. Hence $E^-$
contains no nonzero vertical vectors and,
being Lagrangian, is the graph of a symmetric
negative definite linear operator.

Consequently, $W^-$ is horizontal near $z_*$
and admits the representation
\eqref{local-stable-graph}. The Hessian
$D^2S(x_*)$ is negative definite, so $S$
is strictly concave in a sufficiently small
neighborhood of $x_*$.
\end{proof}

\begin{Small}
\begin{Rem}
The inverse local branch $\Phi^{-1}$ similarly
determines a germ of an unstable invariant
Lagrangian manifold $W^+$ through $z_*$.
Its tangent space has dimension $n$.
\end{Rem}
\par\end{Small}

\begin{Thm}[Local Bellman--Hamiltonian duality]
\label{thm:local-Bellman}
Let $b:X\times X\to\R$ be a smooth jointly concave discrete Lagrangian, and let $U\subseteq\operatorname{int}X$ be a convex open set. Suppose that $S:U\to\R$ is smooth and concave, and that $\f:U\to U$ is a smooth map.

Then the following three conditions concerning $U$, $S$ and $\f$ are equivalent:

\begin{enumerate}
\item \emph{Bellman equation.}
For some constant $\lambda$, the function $S$ satisfies
\begin{equation}\label{local-Bellman}
S(x)+\lambda
=\max_{y\in U}\{b(x,y)+S(y)\},
\qquad x\in U,
\end{equation}
and $\f(x)$ is the maximizing point.

\item \emph{Legendre relations.}
For every $x\in U$,
\begin{equation}\label{local-Bellman-Legendre}
\partial_xb(x,\f(x))=\d S(x),
\qquad
\partial_yb(x,\f(x))=-\d S(\f(x)).
\end{equation}

\item \emph{Hamilton--Jacobi representation.}
The graph $\operatorname{graph}\d S$ is invariant under the transition $\f:U\to U$ associated with the Hamiltonian correspondence $\Theta$, and
\begin{equation}\label{local-Bellman-HJ}
b(x,\f(x))+S(\f(x))-S(x)=\lambda
\end{equation}
for some constant $\lambda$.
\end{enumerate}

Moreover, assume that $x\mapsto b(x,x)$ attains its maximum at $x_*\in\operatorname{int}X$, that $D^2b(x_*,x_*)$ is negative definite, and that $\partial^2_{xy}b(x_*,x_*)$ is nondegenerate. Then the equivalent conditions above are satisfied, as germs near $x_*$, by a smooth strictly concave function $S$, unique up to an additive constant among concave solutions satisfying
\begin{equation}
\d S(x_*)=\partial_xb(x_*,x_*),
\qquad
\f(x_*)=x_*.
\end{equation}
Its differential generates the stable invariant Lagrangian germ $W^-$ of $\Theta$ at $z_*=(x_*,\partial_xb(x_*,x_*))$, and the corresponding spectral constant is $\lambda=b(x_*,x_*)$.
\end{Thm}

\begin{proof}
Assume first that the Bellman equation holds. Since $\f(x)$ is an interior maximizer, the first-order optimality condition gives
\begin{equation}
\partial_yb(x,\f(x))+\d S(\f(x))=0.
\end{equation}
Differentiating \eqref{local-Bellman} with respect to $x$ and using this condition yields
\begin{equation}
\partial_xb(x,\f(x))=\d S(x).
\end{equation}
Thus (1) implies (2).

Conversely, the Legendre relations imply that $\f(x)$ is a stationary point of the function $y\mapsto b(x,y)+S(y)$. Concavity makes it a maximizer. Moreover, differentiating
\begin{equation}
b(x,\f(x))+S(\f(x))-S(x)
\end{equation}
and using both Legendre relations shows that this expression is constant on the connected set $U$. Hence (2) implies (1).

The equivalence of (2) and (3) follows directly from the definition of the Hamiltonian correspondence and the same constancy argument.

Finally, Theorem~\ref{thm:local-stable-Lagrangian} supplies a unique stable invariant Lagrangian germ $W^-=\operatorname{graph}\d S$ through $z_*$, with $S$ smooth and strictly concave. Its local state transition fixes $x_*$, and a sufficiently small forward invariant neighborhood provides the required map $\f:U\to U$. The equivalence just proved yields the localized Bellman equation, with $\lambda=b(x_*,x_*)$.

To establish uniqueness among concave solution germs, observe that their invariant Lagrangian graphs have tangent spaces at $z_*$ on which $Q(\xi,\eta)=\<\eta,\xi\>$ is nonpositive, by concavity of $S$. The tangent space $L$ is also invariant under $A=D\Phi(z_*)$. Since $A$ is hyperbolic, any $A$-invariant subspace decomposes into its stable and unstable spectral parts. A nonzero unstable part of $L$ would contain a nonzero vector of $E^+$, where $Q$ is positive definite by Theorem~\ref{thm:discrete-hyperbolicity}, contradicting $Q|_L\le0$. Hence $L\subseteq E^-$; equality follows from $\dim L=\dim E^-=n$. After shrinking the invariant graph, its induced forward map is contracting in an adapted local norm, since its differential at $x_*$ has only stable eigenvalues. Thus every point of the graph sufficiently close to $z_*$ converges to $z_*$, and uniqueness of the local stable manifold identifies the graph with $W^-$. The generating functions differ only by an additive constant.
\end{proof}

\begin{Small}
\begin{Rem}
The theorem is entirely local: all statements concern germs near $x_*$ and $z_*$. The neighborhoods involved may be reduced as necessary. No global continuation of the pseudocascade or regularity assumptions near $\partial X$ are required.
\end{Rem}
\par\end{Small}

\begin{Small}
\subsection{De-localization and state constraints}

The local Bellman--Hamiltonian duality established above admits a natural continuation problem. In ordinary dynamical systems, local invariant manifolds are extended by iterating the underlying map. For a Hamiltonian pseudocascade, no globally defined evolution map is available in general. Nevertheless, the symplectic correspondence $\Theta$ allows the stable germ to be continued backward along admissible finite trajectories, without requiring a globally defined cascade.

This de-localization problem should be distinguished from the stronger requirement that the continued manifold project onto the entire interior of $X$. The distinction is already familiar from continuous time variational optimization, where state constraints may obstruct the classical Hamiltonian description at the boundary. In discrete time, finite-step transitions introduce additional difficulties of solvability and branch selection. We do not attempt a general theory of these effects here, nor impose assumptions designed to eliminate them throughout $X$.

The distinction is also relevant to the embedding problem discussed in the final section. The symplectic obstruction exhibited there is local: they arise from the linearization at a hyperbolic stationary point and do not depend on any global continuation of the invariant manifold.

We denote by $\Wloc$ a sufficiently small representative of the stable Lagrangian germ $W^-$ constructed in Theorem~\ref{thm:local-stable-Lagrangian}, chosen so that its forward trajectories under the local branch $\Phi$ remain in $\Wloc$ and converge to $z_*$. Using the correspondence $\Theta$, we define
\begin{equation}\label{Wglo}
\Wglo=\bigcup_{k\ge 0}\Theta^{-k}(\Wloc).
\end{equation}
Here $\Theta^{-k}(\Wloc)$ consists of all phase points from which a finite regular trajectory of the Hamiltonian correspondence $\Theta$ reaches $\Wloc$ after $k$ steps, with every underlying transition $(x_t,x_{t+1})$ in $\operatorname{int}X\times\operatorname{int}X$. Thus $\Wglo$ is well defined without assuming that $\Theta$ is a single-valued map or that its iterations exist everywhere. Only regular interior continuations are included; trajectories meeting $\partial X$ are not part of this construction.

The definition of $\Wglo$ is purely set-theoretic. To understand whether it produces a smooth continuation of $\Wloc$, we must distinguish several possible obstructions.

\begin{enumerate}
\item \emph{Missing transitions.}
For a given phase point $z$, there may be no $z'$ such that $(z,z')\in\Theta$, or no $z'$ such that $(z',z)\in\Theta$. Thus a prescribed finite chain may have no continuation in the required direction. In particular, the sets $\Theta^{-k}(\Wloc)$ need not cover any prescribed phase domain.

\item \emph{Multiple continuations.}
Although each local branch of $\Theta$ is a diffeomorphism, several distinct branches may be available at the same phase point. Consequently, a point of $\Wloc$ may have several predecessors, and different finite chains may lead to the same point. Each chain defines a smooth local continuation, but these continuations need not agree.

\item \emph{Loss of horizontality.}
A local inverse symplectomorphism preserves the Lagrangian property but need not preserve transversality to the vertical fibers of $\pi\:X\times\R^{n*}\to X$. Thus a continued Lagrangian manifold may develop vertical tangent vectors and cease to be locally representable as $\operatorname{graph}\d S$.

\item \emph{Multiple momenta over one state.}
Even when every continued piece is horizontal, two distinct phase points $(x,p_1)$ and $(x,p_2)$ may belong to $\Wglo$. Then the horizontal projection is not injective, so the continued set cannot be represented by a single-valued generating function on its entire projection.
\end{enumerate}

The local Legendre condition rules out singularities of individual regular branches, but none of the other obstructions. The Lyapunov monotonicity associated with joint concavity of $b$ provides the additional control needed to address the last two phenomena and to reconcile different continuations.

\begin{Thm}[De-localization of the stable Lagrangian germ]
\label{thm:delocalization}
Let $X\subseteq\R^n$ be a compact convex set with nonempty interior, and let $b:X\times X\to\R$ be a smooth strictly jointly concave discrete Lagrangian satisfying the local Legendre condition. Assume that $x_*\in\operatorname{int}X$ is the unique maximizer of $x\mapsto b(x,x)$ and that
\begin{equation}
D^2b(x_*,x_*)<0.
\end{equation}
Let $\Theta$ be the associated Hamiltonian pseudocascade and $\Wloc$ a sufficiently small representative of its stable Lagrangian germ $W^-$ at $z_*=(x_*,\partial_xb(x_*,x_*))$, as constructed in Theorem~\ref{thm:local-stable-Lagrangian}. Define $\Wglo$ as above in \eqref{Wglo}.

Then the portion of $\Wglo$ over $\operatorname{int}X$ is a smooth horizontal Lagrangian graph over an open set $U\subseteq\operatorname{int}X$ containing $x_*$. In particular, its projection onto $U$ is a diffeomorphism.
\end{Thm}

\begin{proof}
Each finite regular backward continuation of $\Wloc$ is locally a smooth Lagrangian manifold, since the local branches of $\Theta$ are symplectomorphisms.

Every point of $\Wglo$ admits a forward trajectory
which eventually enters $\Wloc$ and converges to $z_*$. By
Lemma~\ref{lem:asymptotic-Lyapunov}, two distinct points
of $\Wglo$ cannot have the same state coordinates; moreover, no regular backward continuation can have a nonzero vertical tangent vector. Indeed, such a vector is transported forward by the differentials of the finitely many regular symplectic branches into the tangent space of $\Wloc$, where it converges to zero under the contracting local branch. The second part of the lemma then gives strict negativity of its initial $Q$, contradicting verticality.
Their horizontal graph representations therefore agree
on overlapping domains, and together form a smooth
Lagrangian graph over the interior of their projection.
\end{proof}

The inclusion $U\subseteq\operatorname{int}X$ may be strict, but equality is possible, as the following examples illustrate.

\begin{Ex}
For an example with full interior coverage, take $X=[-1,1]$ and
\begin{equation}
b(x,y)=-\frac a2x^2-\frac c2(y-x)^2,
\qquad a,c>0.
\end{equation}
The stationary point is $x_*=0$. If $r\in(0,1)$ is the smaller root of
$cr^2-(a+2c)r+c=0$, then the trajectories on the stable graph satisfy $x_{t+1}=rx_t$ and converge to $x_*=0$. Thus $U=(-1,1)$.
\end{Ex}

The projection of the extended stable manifold need not cover the entire interior of $X$, even for a strictly concave quadratic Lagrangian.

\begin{Ex}
Consider
\begin{equation}
X=[-1,10],\qquad b(x,y)=-(x^2+xy+y^2).
\end{equation}
The full Hessian of $b$ is negative definite, and the
mixed Legendre derivative is nonzero. The unique
stationary point is $x_*=0$, while the discrete
Euler--Lagrange equation reads $x_{t-1}+4x_t+x_{t+1}=0$.

Its solutions converging to $x_*$ have the form $x_t=r^t x_0$, where
$r=-2+\sqrt{3}\in(-1,0)$. Such a stable solution remains in $X$ for every $t\ge0$ precisely when $x_0\in[-1,\,2+\sqrt{3}]$. The endpoint $x_0=2+\sqrt3$ has $x_1=-1\in\partial X$ and is therefore excluded from the regular interior continuation.

Indeed, for $x_0>2+\sqrt{3}$ the first transition
requires $x_1<-1$, violating the state constraint.
Consequently, the interior projection of the extended
stable manifold is
$
U=(-1,\,2+\sqrt{3})
\subsetneq\operatorname{int}X
$. Thus even a globally regular quadratic Lagrangian
need not admit total de-localization of its stable
Hamiltonian graph.
\end{Ex}

\end{Small}

\section{The embedding problem}\label{sec:embedding}

The preceding sections revealed substantial differences between the continuous and discrete time Hamiltonian descriptions of stationary optimization, even under strong joint concavity assumptions. These differences become particularly apparent in the presence of state constraints and make a general global embedding of discrete Hamiltonian dynamics into continuous flows rather implausible. Yet they leave open a more modest possibility: near the interior hyperbolic stationary point, where state constraints are inactive and the Hamiltonian pseudocascade is a genuine local symplectic diffeomorphism, could the discrete dynamics be realized as the time-one map of an autonomous Hamiltonian flow? We shall see that even this local embedding is not always possible: obstructions arise already at the linearization of the pseudocascade, independently of state constraints and global continuation problems.

\subsection{Crash course on different embedding problems}

The embedding problem is a vast subject in its own right; here we shall touch upon it only insofar as needed for our purposes.

The syntax of these embedding problems is essentially the same. We consider a cyclic (semi)group generated by an element $f$, usually a transformation of some set $\mathscr S$, and ask for a one-parameter (semi)group $f^t$, with $t\in\R_+$ or $t\in\R$ as appropriate, such that $f=f^1$. One can impose different requirements on the (semi)group, e.g.:
\begin{enumerate}
 \item Each $f^t$ should be a self-map of the same $\mathscr S$;
 \item The correspondence $t\mapsto f^t$ should be continuous in some natural topology, in particular, at $t=0$ when $f^0=\operatorname{id}$;
 \item If $\mathscr S$ is equipped by a certain structure (topological space, smooth manifold etc.) preserved by $f$, then all maps $f^t$ must also preserve this structure.
\end{enumerate}
Obstructions can arise on any level. 

\begin{Small}
\begin{Ex}
Assume that $\mathscr S$ is finite and $f\:\mathscr S\to\mathscr S$ is an \emph{odd} permutation. Then $f^{1/2}$ cannot exist in the class of permutations, because its square must be even.
\end{Ex}

\begin{Ex}
If $\mathscr S$ is a finite-dimensional linear space and $f=A$ is a linear operator. If $A$ is non-invertible, then its embedding into a continuous \emph{semi}group is impossible: because of continuity at $t=0$, the operator $A^{1/n}$ must be invertible for $n$ large enough. This means that $A=\bigl(A^{1/n}\bigr)^n$ also must be invertible. 

In a similar way, an orientation-reversing \emph{real} operator cannot be embedded in a flow: for the same reasons $A^{1/n}$ for large $n$ must be orientation-preserving by continuity, hence $A$ must also be.
\end{Ex}  

\begin{Rem}
One should distinguish embedding from \emph{suspension}. Any diffeomorphism $f:\mathscr S\to\mathscr S$ can be realized as the return map of an autonomous flow on its mapping torus, obtained by identifying $(x,1)$ with $(f(x),0)$ in $\mathscr S\times[0,1]$. This construction enlarges the phase space and does not imply that $f$ can be embedded into a one-parameter flow acting on $\mathscr S$ itself.
\end{Rem}

\par\end{Small}   

\subsection{Embedding of linear operators: the Culver test}
If $A$ is a finite-dimensional linear operator, then its embeddability is equivalent to existence of its matrix logarithm $B=\log A$, the matrix satisfying the equation $\exp B=A$. Here $\exp$ of a matrix is understood in the classical sense, as the result of the substitution of $B$ into the formal power series for the exponent. The $A^t=\exp tB$, $\in\R$, will be the embedding group.

\begin{Prop}
For an invertible $A$ the embedding is always possible over $\C$. Indeed,
it is enough to consider Separately each Jordan block
$J_\l=\lambda E+N$, where $\lambda\ne0$ and $N$ is nilpotent. Then
\begin{equation}\label{matlog}
 \log J_\l=(\log\lambda)E+\log(E+\lambda^{-1}N),
\end{equation}
where the second logarithm is a finite polynomial in $N$. The corresponding one parameter group is $A^t=\exp (t\log A)$, where $\log A$ is assembled of the corresponding Jordan blocks $\log J_\l$ for all eigenvalues of $A$. \qed
\end{Prop}

The real case is substantially more delicate, since logarithms of the negative eigenvalues are non-real and formula \eqref{matlog} becomes illegal over $\R$. Simple preservation of orientation is insufficient for existence of the real logarithm.
The complete criterion for the existence of a real matrix logarithm was
given by Culver \cite{Culver1966}. Two identical negative Jordan blocks can be paired into one real logarithmic block, just as the central symmetry of the 2-plane, $v\mapsto -v$, embeds as a rotation by $\pi$; Culver's condition says precisely that all negative Jordan blocks can be paired in this way.

\begin{Thm}[Culver test]\label{thm:culver}
An invertible real matrix $A$ admits
a real logarithm if and only if, for every negative real eigenvalue
$\lambda$ of $A$, the Jordan blocks $J_\lambda$ of each given size
occur an even number of times. Thus positivity of $\det A$, although
necessary, is not sufficient: the obstruction depends on the Jordan
structure associated with each individual negative eigenvalue. \qed
\end{Thm}

Thus a diagonal $2\times2$ matrix with eigenvalues $-2$ and $-1/2$ admits no real logarithm, as it fails the Culver test, unlike the scalar $2\times2$-matrix $-E$.

\begin{Small}
\subsection{The linear symplectic case}

The embedding problem for symplectic matrices is, in principle, more
restrictive than the general linear one. Recall that the negative of the canonical
symplectic form $\omega=\d p\land\d x$ on $\R^{2n}$ is represented by
$
 J=
 \begin{pmatrix}
  0&E\\
  -E&0
 \end{pmatrix}.
$
The symplectic group and its Lie algebra are
\begin{gather*}
 \Sp(2n,\R)
 =
 \{A\in\GL(2n,\R):A^{\mathsf T}JA=J\},
\\
 \mathfrak{sp}(2n,\R)
 =
 \{K\in\mathfrak{gl}(2n,\R):
 K^{\mathsf T}J+JK=0\}.
\end{gather*}
Thus a symplectic matrix $A$ can be embedded into a linear Hamiltonian
flow precisely when it admits a Hamiltonian logarithm,
$ A=\exp K$ with $K\in\mathfrak{sp}(2n,\R)$.

The existence criterion for such logarithms, originally obtained by
Williamson \cite{Williamson1939}, is slightly more restrictive than the
Culver test for arbitrary real logarithms. In the constructive
formulation of Dieci \cite{Dieci1998}*{Theorem~2.7}, the additional
restrictions concern the symplectic canonical blocks associated with
negative real eigenvalues; the eigenvalue $-1$ requires a separate
condition which has no counterpart in the ordinary Culver criterion
(see also \cite{Dieci1998}*{Remark~2.9}).

These additional restrictions disappear, however, in the hyperbolic
case (no modulus one eigenvalues) which is relevant to us. Indeed, $-1$ cannot then be an eigenvalue, while
the remaining negative symplectic blocks are of type~(a) in the
notation of \cite{Dieci1998}. Such a block contains ordinary Jordan
blocks of the same size corresponding to the reciprocal eigenvalues
$\lambda$ and $\lambda^{-1}$. Hence the pairing condition in
\cite{Dieci1998}*{Theorem~2.7(1)} is exactly the Culver condition.
Consequently, for a hyperbolic matrix $A\in\Sp(2n,\R)$,
\begin{equation*}
 A\text{ admits a real logarithm}
 \quad\Longleftrightarrow\quad
 A\text{ admits a real Hamiltonian logarithm}.
\end{equation*}
Thus in the hyperbolic symplectic case no additional linear
obstruction appears: once the Culver criterion is satisfied, $A$ is
the time-one map of a linear Hamiltonian flow.
\par\end{Small}

\subsection{Non-embeddable linear counterexample}

Consider the quadratic discrete Lagrangian
\begin{equation}\label{eq:culver-lagrangian}
 b(x,y)=-x^2-xy-y^2,\qquad x,y\in\R.
\end{equation}
It is strictly jointly concave (in the Hessian sense) and satisfies the local Legendre
condition at every point. The corresponding Hamiltonian pseudocascade is a genuine
linear cascade, since
\begin{equation*}
 p=-2x-y,\qquad q=x+2y,
 \qquad
 \binom yq=A\binom xp,\quad
 A=\begin{pmatrix}-2&-1\\-3&-2\end{pmatrix}.
\end{equation*}
The matrix $A$ is symplectic and hyperbolic, with two distinct
negative eigenvalues $\lambda_\pm=-2\pm\sqrt3$.
Thus it fails the Culver test and cannot be embedded into a
real linear flow. In fact, it cannot be the time-one map of any smooth
real autonomous local flow near the origin. Indeed, the origin is an
isolated fixed point of $A$, since $1$ is not an eigenvalue. If it
were not an equilibrium of such a flow, its local orbit would give
a nonconstant curve of fixed points of the time-one map, a
contradiction. At an equilibrium, however, the differential of the
time-one map is the exponential of the linearized vector field,
so $A$ would admit a real logarithm.

\begin{Rem}
The Lagrangian \eqref{eq:culver-lagrangian} can be restricted
to any compact symmetric interval $X=[-R,R]$, $R>0$.
The stable and unstable Lagrangian manifolds of the linear
extension are explicitly given by $p=\mp\sqrt3\,x$.
In particular, the stable graph extends over the entire $X$,
and its forward trajectories remain in $X$.
Thus the example satisfies the hypotheses of the
de-localization Theorem \ref{thm:delocalization} and even admits a global horizontal 
stable graph, yet its Hamiltonian pseudocascade is not
locally embeddable into an autonomous smooth flow.
\end{Rem}

\subsection{Embedding of endomorphisms of idempotent semimodules}

We have just seen that in the finite-dimensional \emph{linear} case,
non-existence of logarithms is a well-understood obstruction to
embedding: over $\C$ every invertible matrix admits a logarithm,
while over $\R$ the obstruction is characterized by the Culver test.
One might hope that semimodules over the idempotent semiring
$\Rbar=\R\cup\{-\infty\}$, equipped with
$\oplus=\max_\R$ and $\otimes=+_\R$
(see Definition~\ref{def:semiring}), would be more amenable to embedding
Bellman operators.

\begin{Small}
\begin{Ex}
Consider the ``diagonal'' Lagrangian
\begin{equation}\label{idem-diagonal}
b(x,y)=
\begin{cases}
\l(x), &x=y,\\
-\infty, &x\ne y,
\end{cases}
\qquad \l(x)\in\R.
\end{equation}
The corresponding Bellman operator $\B$ embeds into the
one-parameter group $\B^t$, $t\in\R$, whose kernel is
$b^t(x,y)=t\l(x)$ for $x=y$ and $-\infty$ otherwise.
\end{Ex}
\par\end{Small}

However, even for a two-point state space $|X|=2$,
there exist Bellman operators without a max-plus square root,
and hence without an embedding into a one-parameter semigroup.

\begin{Lem}\label{lem:idem-square}
If $|X|=2$ and $U$ is a max-plus matrix, then
\begin{equation}\label{idem-square}
U^{\otimes2}=
\begin{pmatrix}p&q\\r&s\end{pmatrix}
\quad\Longrightarrow\quad q+r\le p+s.
\end{equation}
\end{Lem}

\begin{Small}
\begin{proof}
Write
\begin{equation*}
U=\begin{pmatrix}a&b\\c&d\end{pmatrix},
\qquad m=\max(a,d).
\end{equation*}
Then
\begin{equation*}
U^{\otimes2}=
\begin{pmatrix}
\max(2a,b+c)&b+m\\
c+m&\max(b+c,2d)
\end{pmatrix}.
\end{equation*}
Thus $q+r=b+c+2m$, while
$b+c\le\min(p,s)$ and
$2m=\max(2a,2d)\le\max(p,s)$.
Consequently $q+r\le p+s$.
\end{proof}

\begin{Cor}
The Bellman operator with the max-plus matrix
\begin{equation*}
B=\begin{pmatrix}-1&-1\\0&-1\end{pmatrix}
\end{equation*}
has no max-plus square root.
\end{Cor}

\begin{proof}
Indeed, $q+r=-1>-2=p+s$.
\end{proof}
\par\end{Small}

\subsection{Maslov dequantization at work again}

The preceding example may appear artificial, but the inequality
\eqref{idem-square} has a natural interpretation as the idempotent
counterpart of the nonnegativity of the determinant of a matrix
square. Although subtraction is unavailable in the max-plus
semiring, the comparison between the two determinant monomials
survives:
$
q\otimes r\le p\otimes s,
$
or equivalently $q+r\le p+s$.

This observation extends to arbitrary finite dimensions by Maslov
dequantization (see \secref{sec:dequantization}).

For an ordinary nonnegative $n\times n$ matrix $V=(v_{ij})$,
the Cauchy expansion of its determinant can be separated into
the contributions of even and odd permutations:
\begin{equation}\label{tropical-square}
\det V=D_+(V)-D_-(V),
\end{equation}
where the sign $_+$ corresponds to even permutations $\sigma$ and the sign $_-$ to the odd ones:
\begin{equation*}
D_\pm(V)=
\sum_{\substack{\sigma\in S_n\\
                 \sigma\text{ even/odd}}}
\prod_{i=1}^n v_{i\sigma(i)}.
\end{equation*}
In particular, if $V$ is the square of a real matrix, then
$D_+(V)\ge D_-(V)$.

For a max-plus matrix $V=(v_{ij})$, define the corresponding
tropical expressions by
\begin{equation*}
D_\pm(V)=
\max_{\substack{\sigma\in S_n\\
                 \sigma\text{ even/odd}}}
\sum_{i=1}^n v_{i\sigma(i)}
=
\bigoplus_{\substack{\sigma\in S_n\\
                     \sigma\text{ even/odd}}}
\bigotimes_{i=1}^n v_{i\sigma(i)}.
\end{equation*}
Here, as usual, sums and maxima are understood in the extended
max-plus arithmetic.

\begin{Thm}\label{thm:tropical-square}
If $V=U^{\otimes2}$ for a finite-dimensional max-plus matrix $U$,
then
\begin{equation*}
D_-(V)\le D_+(V).
\end{equation*}
\end{Thm}

\begin{proof}
For $h>0$, associate with $U=(u_{ij})$ the nonnegative real
matrix
\begin{equation*}
A_h=(e^{u_{ij}/h}),
\qquad e^{-\infty}=0,
\end{equation*}
and put $V_h=A_h^2$. Since
$\det V_h=(\det A_h)^2\ge0$, the ordinary determinant
expansion gives
\begin{equation*}
D_+(V_h)\ge D_-(V_h),
\end{equation*}
where $D_\pm$ denote the ordinary sums of permutation monomials.

By the elementary log-sum-exp limit,
\begin{equation*}
h\log(V_h)_{ij}\longrightarrow
\max_k(u_{ik}+u_{kj})=v_{ij}
\end{equation*}
as $h\to0^+$, with $h\log0=-\infty$.
Applying the same limit to the two nonnegative sums
$D_\pm(V_h)$ yields
\begin{equation*}
h\log D_\pm(V_h)\longrightarrow D_\pm(V).
\end{equation*}
Passing to the limit in the preceding inequality proves
the claim.
\end{proof}

\begin{Cor}
A finite-dimensional max-plus matrix $V$ with
$D_-(V)>D_+(V)$ cannot be embedded into a one-parameter
semigroup of max-plus matrices.
\end{Cor}

\begin{proof}
Such an embedding would provide a max-plus square root
$U=V^{1/2}$, contradicting Theorem~\ref{thm:tropical-square}.
\end{proof}

\section*{In summa}
\begin{Small}

The first two sections revisit the theory of stationary
infinite-horizon optimization as it stood, for the most part,
some forty years ago. Section~\secref{sec:infinite-extremals} recalls how the problem of
defining infinite extremals in discrete time leads to the
spectral equation for Bellman operators. Section~\secref{sec:variations} turns to
continuous time, where the classical Euler--Lagrange and
Hamiltonian equations, together with the Bellman semigroup,
provide a geometric interpretation of the missing boundary
conditions at infinity in terms of invariant Lagrangian
manifolds.

Section~\secref{sec:hamiltonian-cascades} surveys the symplectic interpretation of discrete
Euler--Lagrange equations developed in the intervening decades,
notably through the discrete Legendre transformations of
Marsden and West and their successors. We pay particular
attention to the domains and ranges of these transformations:
the resulting Hamiltonian dynamics is generally a
pseudocascade, rather than a globally defined map which can
be iterated indefinitely.

Section~\secref{sec:hyperbolic} establishes a local counterpart of the continuous-time
hyperbolic picture. For strictly jointly concave discrete
Lagrangians, a hyperbolic stationary point possesses a stable
invariant Lagrangian germ whose generating function solves a
localized Bellman spectral equation. The germ admits a
continuation by the Hamiltonian correspondence, but this
does not in general guarantee a global graph over the entire
state space.

Finally, Section~\secref{sec:embedding} shows that even this close local parallel
between discrete and continuous Hamiltonian dynamics does not
ensure embeddability. A strictly jointly concave quadratic
Lagrangian already produces a hyperbolic symplectic cascade
whose linearization fails the Culver test. Independently, in
the max-plus setting, a tropical analogue of the determinant
obstruction shows that certain Bellman operators do not even
admit square roots. Thus neither the symplectic nor the
idempotent algebraic structure provides an automatic passage
from discrete to continuous time.

\subsection*{Some open problems}

The results obtained above, both positive and negative, are far from exhausting the embedding problem for stationary dynamic optimization. We conclude with two examples of questions left open, without attempting \emph{to embrace the unembraceable} (Kozma Prutkov, cited after Kantorovich \cite{Kantorovich1976}).

\subsubsection*{Different flavors of the embedding}

In fact, the embedding problem for stationary dynamic optimization can be formulated in at least three different ways:
\begin{enumerate}
 \item[(B)] Given a Bellman operator $\B\:\Rbar(X)\to\Rbar(X)$, find a one-parameter \emph{semi}group $\{\B^t\}_{t\in\R_+}$, continuous in $t$ with respect to a suitable operator topology, such that $\B=\B^1$ and $\B^0=\operatorname{id}$. 

 \item[(H)] Given a Hamiltonian \emph{invertible} cascade $\Phi\:M\to M$ on a smooth symplectic manifold with boundary, construct a \emph{continuous} one-parameter \emph{group} $\{\Phi^t\}_{t\in\R}$ on this manifold such that $\Phi^1=\Phi$.

 \item[(L)] Starting from $\Phi$ as above, construct a \emph{differentiable} one-para\-meter group $\{\Phi^t\}_{t\in\R}$ as above, generated by a smooth Hamiltonian vector field $V=V_H$ on $M$, and construct the respective con\-tinuous-time Lagrangian $L(x,v)$ so that
     $$
     b(x_0,y_1)=\max \int_0^1 L(u(t),\dot u(t))\,\d t,
     \qquad u(0)=x_0,\quad u(1)=y_1,
     $$
     where the maximum is taken over all absolutely continuous trajectories connecting the endpoints.
\end{enumerate}
As before, positive answers can realistically be expected only for sufficiently special Bellman operators and  Hamiltonian cascades. The distinction between (H) and (L) is essential: the former requires only continuity in the time parameter, whereas the latter requires differentiability and a Hamiltonian generator, together with a compatible variational representation.

\begin{Rem}
Problem (B) is formulated for semigroups rather than groups
for a good reason: already for finite $X$, every invertible
max-plus matrix is a generalized permutation matrix
\cites{butkovic,vasya}. Thus invertibility leaves only
deterministic transitions, with no genuine optimization.
\end{Rem}

\subsubsection*{Nonlinear obstructions to embeddability}

Construct a smooth, strictly jointly concave discrete Lagrangian
whose Hamiltonian pseudocascade has a hyperbolic stationary point
and whose linearization admits a \emph{real} Hamiltonian logarithm,
but which cannot be embedded, even locally near that point,
into an autonomous Hamiltonian flow.

\par\end{Small}

\end{document}